\documentclass{siamltex}
\usepackage{cases}
\usepackage{amssymb}

\newtheorem{them}{Theorem}[section]
\newtheorem{lem}{Lemma}[section]

\newtheorem{coro}{Corollary}[section]
\newtheorem{pro}{Proposition}[section]

\begin{document}

\title{Fully piecewise linear vector optimization problem
\footnotemark[1]}
\author{Xi Yin Zheng\footnotemark[2] \and Xiaoqi Yang\footnotemark[3]}
\maketitle
\begin{abstract}
We distinguish two  kinds of piecewise linear functions and provide an interesting representation for a piecewise linear function between two normed spaces. Based on such a representation, we study a fully piecewise linear vector optimization  (PLP) with the objective and constraint functions being piecewise linear. We divide (PLP) into some linear subproblems and structure a finite dimensional reduction method to solve (PLP). Under some mild assumptions, we prove that the Pareto (resp. weak Pareto) solution set of (PLP) is the union of finitely many generalized polyhedra (resp. polyhedra), each of which is contained in  a Pareto (resp. weak Pareto) face of some linear subproblem.
Our main results are even new in the linear case and  further generalize Arrow, Barankin and Blackwell's classical results on
linear vector optimization problems in the framework of finite dimensional spaces.
\end{abstract}
\renewcommand{\thefootnote}{\fnsymbol{footnote}}

\footnotetext[1]{This research was supported by  the National Natural Science
Foundation of P. R. China (Grant No. 11771384) and the Research Grants Council of Hong Kong (PolyU 152128/17E).}
\footnotetext[2]{Department of Mathematics, Yunnan University,
Kunming 650091, P. R. China (xyzheng@ynu.edu.cn).}
\footnotetext[3]{Department of Applied Mathematics, The Hong Kong Polytechnic University, Hong Kong, P. R. China (xiao.qi.yang@polyu.edu.hk).}


\vskip1cm

\section{Introduction}
Though vector optimization is often  encountered in theory and practical application,  the study of nonlinear vector optimization is far from systemic (possibly because the vector ordering is  much more complicated than the scalar one).  On the other hand,  linear vector optimization has been well studied (cf. \cite{Ar,BS,GR,jahn,Luc,Luc2,Perez,Ze} and the references therein).  In particular,
in the finite-dimensional case,  Arrow, Barankin and Blackwell \cite{ABB}  established the structure of the Pareto solution set and  weak Pareto solution set of a linear vector optimization problem. However the linearity assumption is  quite restrictive  in both theory and application. To overcome the restriction of linearity, one sometimes adopts the piecewise linear functions (cf. \cite{grv,FMY, YY, Zh}). The family of all piecewise
linear functions is much larger than that of all linear functions and  there exists a   wide
class of functions that can be approximated by piecewise linear functions. Therefore, from the viewpoint of  theoretical interest as well as for applications,
it is important to study  piecewise linear problems. Given two normed spaces $X$ and $Y$,  the following piecewise linearity of a vector-valued function $f:X\rightarrow Y$ was adopted in the literature (cf. \cite{YY, ZY}): {\it there exist finitely many polyhedra $\Lambda_1,\cdots,\Lambda_m$ in the product $X\times Y$ such that}
\begin{equation}\label{pl1}
{\rm gph}(f):=\{(x,f(x)):\;x\in X\}=\bigcup\limits_{i=1}^m\Lambda_i.
\end{equation}
Throughout this paper, we will use $\mathcal{P}(Z)$ to denote the family of all polyhedra in a normed space $Z$.
Another kind of piecewise linearity for a function $f$ is as follows:  {\it there exist $T_i\in\mathcal{L}(X,Y)$, $P_i\in\mathcal{P}(X)$ and $b_i\in Y$ ($i=1,\cdots,m$) such that}
\begin{equation}\label{pl2}
X=\bigcup\limits_{i=1}^mP_i \;\;{\rm and}\;\;f(x)=T_i(x)+b_i\quad\forall x\in P_i,\;i=1,\cdots,m,
\end{equation}
where $\mathcal{L}(X,Y)$ denotes the space of all  continuous linear operators from $X$ to $Y$. For convenience, let $\mathcal{PL}_1(X,Y)$ (resp. $\mathcal{PL}(X,Y)$) denote the family of all piecewise linear functions from $X$ to $Y$ in the sense of (\ref{pl1}) (resp. (\ref{pl2})). It is clear that $\mathcal{L}(X,Y)$ is always contained in $\mathcal{PL}(X,Y)$; however if $Y$ is infinite dimensional then every linear operator in $\mathcal{L}(X,Y)$ must not be in $\mathcal{PL}_1(X,Y)$. This motivates us to study the relationship between $\mathcal{PL}_1(X,Y)$ and $\mathcal{PL}(X,Y)$. To do this, we first consider polyhedra in normed spaces. In Section 2, we provide several properties on
polyhedra in normed spaces. In particular, with the help of the notion of a prime generator group of a polyhedron (cf. \cite{BCL, Te, KTZ}), we establish some results on the  maximal faces of a polyhedron, which not only play a key role in the proof of the main theorem on piecewise linear functions but also should be valuable by themselves. In Section 3, using the results obtained in Section 2, we prove that
$$
{\rm dim}(Y)<\infty\Leftrightarrow \mathcal{PL}_1(X,Y)=\mathcal{PL}(X,Y)\;\;\;{\rm and}\;\;\;
{\rm dim}(Y)=\infty\Leftrightarrow \mathcal{PL}_1(X,Y)=\emptyset.
$$
As one of the mains results, we  prove by using the Fubibi theorem on Lebesgue's measure that for each $f\in\mathcal{PL}(X,Y)$ there exist two closed subspaces $X_1$ and $X_2$ of $X$,  a closed subspace $Y_2$ of $Y$, $T\in\mathcal{L}(X_1,Y)$ and  $g\in\mathcal{PL}_1(X_2,Y_2)$ such that  $X=X_1\oplus X_2$, ${\rm dim}(X_2)<\infty$, ${\rm dim}(Y_2)<\infty$  and
$$f(x_1+x_2)=Tx_1+g(x_2)\;\;\forall (x_1,x_2)\in X_1\times X_2.$$

In Section 4, we consider  a fully piecewise linear vector optimization  problem in the framework of general normed spaces. In the case when $f\in\mathcal{PL}(X,Y)$  and $\varphi_j\in\mathcal{PL}(X,\mathbb{R})$ ($j\in\overline{1l}:=\{1,\cdots,l\}$), we  study the structure of the (weak) Pareto solution set of the following fully piecewise linear vector optimization problem
$$C-{\rm Min}f(x)\;\;{\rm subject\;to}\;\varphi_j(x)\leq0,\;j=1,\cdots,l,\leqno{\rm (PLP)}$$
where $C$ is a closed convex cone in $Y$.
Let $\leq_C$ denote the preorder induced by $C$
in $Y$, that is,
for $y_1,y_2\in Y$, $y_1\leq_C y_2\Leftrightarrow y_2-y_1\in C$. When the interior ${\rm int}(C)$ of $C$ is nonempty,   $y_1<_Cy_2$ is defined as
$y_2-y_1\in{\rm int}(C)$.
For a subset $\Omega$ of $Y$ and a point $\omega$ in $\Omega$, we say that $\omega$ is a Pareto efficient point of $\Omega$ (with respect to $C$), denoted by $\omega\in E(\Omega,C)$, if there is no element $v\in \Omega\setminus\{\omega\}$ such that
$v\leq_C\omega$. In the case when ${\rm int}(C)\not=\emptyset$, we say that  $\omega$ is  a weak Pareto efficient point of $\Omega$, denoted by $\omega\in{\rm WE}(\Omega,C)$, if there is no element $v\in \Omega$
such that $v<_C\omega$.

Let $A$ denote the feasible set of fully piecewise linear vector optimization problem (PLP), that is,
\begin{equation}\label{Fe}
A:=\{x\in X:\;\varphi_1(x)\leq0,\cdots,\varphi_l(x)\leq0\}.
\end{equation}
We say that $\bar{x}\in A$  is  a Pareto (resp. weak Pareto) solution of (PLP) if  $f(\bar{x})\in {\rm E}(f(A),C)$
(resp. $f(\bar{x})\in{\rm WE}(f(A),C)$).

For each $i\in\overline{1m}$, let
\begin{equation}\label{wp2'}
A_i:=\{x\in P_i:\;\langle x_{ij}^*,x\rangle\leq c_{ij}\;\;\forall j\in\overline{1l}\}.
\end{equation}
To study fully piecewise linear problem (PLP), we also consider the following  linear subproblems
$$C-\min T_ix+b_i\;\;{\rm subject\;to}\;x\in A_i.\leqno{{\rm (LP)}_{i}}$$
Recall that a weak Pareto face (resp. Pareto face) $F$ of linear problem (LP)$_{i}$ is a face of $A_{i}$ such that each point in  $F$ is a weak Pareto solution (resp. Pareto solution) of (LP)$_{i}$.

In the case of finite dimensional spaces, the following well known result on the solution sets for linear vector optimization problems is based on the pioneering work by Arrow et al. \cite{ABB}  (also see \cite[Theorem 3.3]{Luc} and \cite[Theorems 4.1.20 and 4.3.8]{Luc2})

{\bf Theorem ABB.} {\it
Let $X=\mathbb{R}^p$, $Y=\mathbb{R}^q$, $C=\mathbb{R}^q_+$, $f(x)=T(x)+b$ and $\varphi_k(x)=\langle x_k^*,x\rangle+r_k$ for some $T\in\mathcal{L}(X,Y)$, $x_k^*\in X^*=\mathcal{L}(X,\mathbb{R})$ and $(b,r_k)\in Y\times\mathbb{R}$ ($k\in\overline{1m}$). Then the Pareto solution set and  weak Pareto solution set of (PLP) are the union of finitely many faces of $A$, where $A$ is as in (\ref{Fe})}.

In the case when the objective $f$ is further piecewise linear, several authors studied the structure of the Pareto solution set and weak Pareto soluiton set and proved that if the objective $f$ is restricted in $\mathcal{PL}_1(X,Y)$ and each $\varphi_k$ is  linear then the weak Pareto solution set of the corresponding piecewise linear  problem (PLP) is the union of finitely many polyhedra, while its Pareto solution set is the union of generalized polyhedra (cf. \cite{ZY, YY, Ze, FMY} and the references therein).

For piecewise linear functions  $f$ and  $\varphi_j$ appearing in fully piecewise linear vector optimization problem (PLP),  based on Proposition \ref{pro2.6} and its proof (an elementary method), we can select   $P_i\in\mathcal{P}(X)$, $T_i\in\mathcal{L}(X,Y)$,  $u_{ik}^*,x_{ij}^*\in X^*$, $b_i\in Y$ and $t_{ik},c_{ij}\in\mathbb{R}$  ($i=1,\cdots,m$, $k=1,\cdots,q_i$ and $j=1,\cdots,l$)  such that
\begin{equation}\label{Poi}
P_i=\{x\in X:\;\langle u_{ik}^*,x\rangle\leq t_{ik}\;\;\forall k\in\overline{1q_i}\},
\end{equation}
\begin{equation}\label{wp1'}
X=\bigcup\limits_{i=1}^m P_i,\;{\rm int}(P_i)\not=\emptyset,\;P_i\cap{\rm int}(P_{i'})=\emptyset\quad\forall i,i'\in\overline{1m}\;{\rm with}\;i\not=i'
\end{equation}
and
\begin{equation}\label{wp2}
f|_{P_i}=T_i|_{P_i}+b_i\;\;{\rm and}\;\;\varphi_j|_{P_i}=x_{ij}^*|_{P_i}-c_{ij}\quad\forall (i,j)\in\overline{1m}\times\overline{1l}.
\end{equation}
In order to obtain exact formulas for optimal value sets and solution sets of  (PLP), we  structure the following  procedures: \\
{\bf Step 1} (Decomposing the space $X$):  Let
\begin{equation}\label{X1}
X_1:=
\bigcap\limits_{i=1}^{m}\bigcap\limits_{(j,k)\in\overline{1l}\times\overline{1q_i}}\mathcal{N}(x_{ij}^*)\cap\mathcal{N}(u^*_{ik}),
\end{equation}
where $\mathcal{N}(x_{ij}^*):=\{x\in X:\langle x_{ij}^*,x\rangle=0\}$ is the null space of $x_{ij}^*$;
namely, $X_1$ is  the solution space of the following system of linear equations
$$\langle u_{ik}^*,x\rangle=\langle x_{ij}^*,x\rangle=0,\;i=1,\cdots,m,\;j=1,\cdots,l,\;k=1,\cdots,q_i.$$
Take a maximal linearly independent subset $\{e_1^*,\cdots,e^*_\nu\}$  of the finite set $\{u^*_{ik},x_{ij}^*:\;i\in\overline{1m},\;j\in\overline{1l},\;k\in\overline{1q_i}\}$. Then,  for each $\iota\in\overline{1\nu}$, the following system of linear equations
$$\langle e_\iota^*,x\rangle=1\;\;\;{\rm and}\;\;\; \langle e^*_{\iota'},x\rangle=0\;\;\forall\iota'\in\overline{1\nu}\setminus\{\iota\}$$
is solvable; take a solution $h_\iota$ of this system of linear equations.
Let
\begin{equation}\label{X2}
X_2:={\rm span}\{h_1,\cdots,h_\nu\}=\left\{\sum\limits_{\iota=1}^\nu t_\iota h_\iota:\;t_1,\cdots,t_\nu\in\mathbb{R}\right\}
\end{equation}
(in particular, $X_2={\rm span}\{e_1^*,\cdots,e_\nu^*\}$ when $X$ is a Hilbert space).
Then
\begin{equation}\label{6.4}
X=X_1+X_2\;\;\;{\rm  and}\;\;\; X_1\cap X_2=\{0\}.
\end{equation}
{\bf Step 2} (Constructing finite dimensional subspace $Z$ of $Y$): Thanks to Theorem \ref{thm3.0}  and (\ref{6.4}),
\begin{equation}\label{6.5}
\hat T:=T_1|_{X_1}=T_2|_{X_1}=\cdots=T_m|_{X_1}.
\end{equation}
Let $D$ denote the finite set $\bigcup\limits_{i=1}^m\{T_i(h_1),\cdots,T_i(h_\nu),\,b_i\}$ and take $u_1,\cdots,u_\varsigma$ in $D$ with $\varsigma$ being the maximal integer such that $u_1\in D\setminus \hat T(X_1)$,
$$u_2\in D\setminus (\hat T(X_1)+{\rm span}\{u_1\}),\cdots,u_{\varsigma}\in D\setminus (\hat T(X_1)+{\rm span}\{u_1,\cdots,u_{\varsigma-1}\}),$$
where $X_1$ and $h_1,\cdots,h_\nu$ are as in Step 1.
Let
\begin{equation}\label{Z}
Z:={\rm span}\{u_1,\cdots,u_\varsigma\}.
\end{equation}
Clearly, $Z$ is a subspace of $Y$ such that
\begin{equation}\label{W4.1}
{\rm dim}(Z)=\varsigma\;\;\;{\rm and}\;\;\;\hat T(X_1)\cap Z=\{0\}.
\end{equation}
Let $\Pi_Z$ denote the projection from $\hat T(X_1)\oplus Z$ onto $Z$, that is,
\begin{equation}\label{Pi}
\Pi_Z(y+z):=z\quad\forall (y,z)\in \hat T(X_1)\times Z,
\end{equation}
and let $C_Z$ be a convex cone in the finite dimensional space $Z$ defined by
\begin{equation}\label{CZ0}
C_Z:=\Pi_Z((\hat T(X_1)\oplus Z)\cap C).
\end{equation}
{\bf Step 3} (Exact formulas for weak Pareto optimal value set and weak Pareto  set of (PLP)): For each $i\in\overline{1m}$, let
\begin{equation}\label{hatA}
\hat A_i:=\{x_2\in\hat P_i:\,\langle x_{ij}^*,x_2\rangle\leq c_{ij}\;\;\forall j\in\overline{1l}\},
\end{equation}
where $\hat P_i:=\{x_2\in X_2:\,\langle u_{ik}^*,x_2\rangle\leq t_{ik}\;\;\forall k\in\overline{1q_i}\}$.
The weak Pareto optimal value set ${\rm WE}(f(A),C)$ and weak Pareto solution set $S^w$ of (PLP) can be formulized as follows:\\
(i) If $(\hat T(X_1)\oplus Z)\cap{\rm int}(C)=\emptyset$ then ${\rm WE}(f(A),C)=f(A)$ and $S^w=A$.\\
(ii) If $(\hat T(X_1)\oplus Z)\cap{\rm int}(C)\not=\emptyset$ then
$${\rm WE}(f(A),C)=\hat T(X_1)+\bigcup\limits_{i=1}^m\hat V_i^w\;\;{\rm and}\;\; S^w=X_1+\bigcup\limits_{i=1}^m\hat A_i\cap (\Pi_Z\circ T_i)^{-1}(\hat V_i^w-\Pi_Z(b_i)),$$
where $\hat V_{i}^w:=\Pi_Z\big(T_i(\hat A_{i})+b_i\big)\setminus\left(\bigcup\limits_{i'\in\overline{1m}}\Pi_Z\big(T_{i'}(\hat A_{i'})+b_{i'}\big)+(\hat T(X_1)\oplus Z)\cap{\rm int}(C)\right)$ (thanks to Theorems \ref{thm5.1} and \ref{pro5.2} and Lemma \ref{lem5.1}).

Similarly,  with Corollary \ref{coro5.1} and Proposition \ref{pro5.3} replacing Theorems \ref{thm5.1} and \ref{pro5.2},  we can also obtain the formulas for the Pareto optimal value set  and Pareto solution set of (PLP). Based on the above procedures, we establish the structure theorems for  Pareto solution sets and optimal value sets of (PLP).

\section{Polyhedra in normed  spaces} Let $Z$ be a normed space with the dual space $Z^*$. Recall (cf.\cite{A,Ro}) that a subset $P$ of $Z$ is a (convex)  polyhedron  if there exist
$u_1^*,\cdots,u_m^*\in Z^*$ and $s_1,\cdots,s_m\in \mathbb{R}$ such that
\begin{equation}\label{Po2.1}
P=\{x\in Z:\;\langle u_i^*,x\rangle\leq s_i,\;i=1,\cdots,m\}.
\end{equation}
An exposed face of $P$ is a set $F$  such that
$$F=\{u\in P:\;\langle x^*,u\rangle=\sup\limits_{x\in P}\langle x^*,x\rangle\}$$
for some $x^*\in Z^*$ (cf. \cite[P.162]{Ro}). It is known that each polyhedron has finitely many exposed faces.
We say that a subset $\tilde{P}$ of $Z$ is a  generalized polyhedron if there exist a polyhedron $P$ in $Z$, $v_1^*,\cdots,v_k^*\in Z^*$ and $t_1,\cdots,t_k\in\mathbb{R}$ such that
$$\tilde{P}=P\cap\{z\in Z:\;\langle v_i^*,z\rangle<t_i,\;1\leq i\leq k\}.$$
Given $z^*\in Z^*\setminus\{0\}$, let $\mathcal{N}(z^*)$ denote the null space of $z^*$, that is,
$$\mathcal{N}(z^*):=\{z\in Z:\langle z^*,z\rangle=0\}.$$
Then $\mathcal{N}(z^*)$ is a closed subspace of $Z$ with codimension ${\rm codim}(\mathcal{N}(z^*))=1$.

For a convex set $K$ in $Z$, let ${\rm rec}(K)$ denote the recession cone of $K$, that is, ${\rm rec}(K):=\{h\in Z:\;K+\mathbb{R}_+h\subset K\}$, and recall that linearity space ${\rm lin}(K)$ of $K$ is defined as $(-{\rm rec}(K))\cap{\rm rec}(K)$. Then ${\rm lin}(K)$ is the large subspace contained in ${\rm rec}(K)$.  It is clear that if $P$ is a polyhedron in $Z$ as in (\ref{Po2.1}) then
\begin{equation}\label{L2.2}
{\rm rec}(P)=\{x\in Z:\;\langle u_i^*,x\rangle\leq0,\;i=1,\cdots,m\}\;\;{\rm and}\;\;{\rm lin}(P)=\bigcap\limits_{i=1}^m\mathcal{N}(u_i^*).
\end{equation}
Recall that a normed space $Z$ is a direct sum of its two closed subspaces $Z_1$ and $Z_2$, denoted by $Z=Z_1\oplus Z_2$, if $Z_1\cap Z_2=\{0\}$ and $Z=Z_1+Z_2$.
It is easy to verify that if $Z=Z_1\oplus Z_2$ then for each $z\in Z$ there exists a unique $(z_1,z_2)\in Z_1\times Z_2$ such that $z=z_1+z_2$ and the projection mapping $\Pi_{Z_2}:\,Z=Z_1\oplus Z_2\rightarrow Z_2$ is linear, where
\begin{equation}\label{P2.1}
\Pi_{Z_2}(z_1+z_2):=z_2\quad\forall (z_1,z_2)\in Z_1\times Z_2.
\end{equation}
It is known that if $Q$ is a polyhedron in  $Z_1\oplus Z_2$ then $\Pi_{Z_2}(Q)$ is a polyhedron in $Z_2$ (cf. \cite[Theorem 19.3]{Ro} and the following Proposition 2.1).

For a convex  set $C$ in $Z$, let ${\rm int}(C)$ (resp. ${\rm rint}(C)$) denote the interior (relative interior) of $C$. It is known that if ${\rm dim}(Z)<\infty$ and $C\not=\emptyset$ then ${\rm rint}(C)\not=\emptyset$. Throughout, let $\mathbb{N}$ denote the set of all natural numbers and
$$\overline{1m}:=\{1,\cdots,m\}\quad\forall m\in\mathbb{N}.$$
Now we provide some results on polyhedra which are useful for our analysis later.
\begin{pro}\label{pro2.1}
Let $(z_1^*,s_1),\cdots,(z_m^*,s_m)\in Z^*\times\mathbb{R}$ and $P:=\{z\in Z:\;\langle z_i^*,z\rangle\leq s_i\;\;\forall i\in\overline{1m}\}$.
Let $Z_1$ and $Z_2$ be two closed subspaces of $Z$  such that
\begin{equation}\label{1}
Z_1\subset \bigcap\limits_{i=1}^m\mathcal{N}(z_i^*),\;{\rm dim}(Z_2)={\rm codim}(Z_1)<\infty\;\;{\rm and}\;\;Z=Z_1\oplus Z_2.
\end{equation}
Then
\begin{equation}\label{2}
P=Z_1+\hat P\;\;{\rm and}\;\;{\rm rint}(P)=Z_1+{\rm rint}(\hat P),
\end{equation}
where $\hat P:=\{z\in Z_2:\;\langle z_i^*,z\rangle\leq s_i,\;i=1,\cdots,m\}$.
\end{pro}

The first equality in (\ref{2}) is a slight variant of \cite[Lemma 2.1]{Zh} and can be proved similar to the proof of \cite[Lemma 2.1]{Zh}, while the second equality in (\ref{2}) is immediate from the following observation: there exists $L\in(0,\;+\infty)$ such that  $L(\|z_1\|+\|z_2\|)\leq \|z_1+z_2\|$ for all $(z_1,z_2)\in Z_1\times Z_2$ and
the affine subspace ${\rm aff}(Z_1+\hat P)$ is equal to $Z_1+{\rm aff}(\hat P)$ (thanks to  (\ref{1}) and the definition of $\hat P$).

From Proposition 2.1, one can see that many properties on polyhedra established in the finite dimension case also hold in the infinite dimension one. In particular,  the following corollaries are  consequences of Proposition \ref{pro2.1} and  \cite[Corollary 6.5.1]{Ro}.

\begin{coro}\label{coro2.1}
Let  $\{(u_1^*,s_1),\cdots,(u_n^*,s_n)\}$ and $P$ be as in Proposition \ref{pro2.1}. Then
\begin{equation}\label{ri2.18}
{\rm rint}(P)=\{z\in Z:\;\langle u^*_i,z\rangle<s_i,\;i\in\overline{1n}\setminus \bar I_P\}\cap\bigcap\limits_{i\in \bar I_P}F_i,
\end{equation}
where $\bar I_P:=\{i\in\overline{1n}:\;\langle u_i^*,z\rangle=s_i\;{\rm for\;all}\;z\in P\}$ and $F_i:=\{z\in Z:\;\langle u_i^*,z\rangle=s_i\}$.
\end{coro}

\begin{coro}\label{pro2.4}
Let $Z_1$ and  $Z_2$ be two closed subspaces of $Z$ such that
\begin{equation}\label{2.20'}
Z=Z_1\oplus Z_2\;\;\;{\rm and}\;\;\;{\rm dim}(Z_2)<\infty.
\end{equation}
Let $\hat P$ be a polyhedron in $Z_2$ and $\hat F$ be a subset of $\hat P$. Then $\hat F$ is an exposed face of $\hat P$ if and only if $Z_1+\hat F$ is an exposed face of the polyhedron $Z_1+\hat P$ in $Z$.
\end{coro}

The following proposition is known   and useful for us (cf. \cite[Lemma 2.2]{Zh}).

\begin{pro}\label{pro2.2}
Let $P_1$ and $P_2$ be  two polyhedra (resp. generalized polyhedra) in $Z$. Then $P_1+P_2$ and $P_1\cap P_2$ are
polyhedra (resp. generalized polyhedra).
\end{pro}

Note that a closed subspace of $Z$ is not necessarily a polyhedron in $Z$. In fact, it is easy to verify that   a closed subspace $E$ of $Z$ is a polyhedron in $Z$ if and only if its codimension ${\rm codim}(E)$ is finite. Note that if $E$ is a closed subspace of $Z$ with ${\rm codim}(E)<+\infty$ and if $H$ is a subspace of $E$ then $E+H$ is a closed subspace of $Z$ with ${\rm codim}(E+H)<+\infty$. The following proposition can be easily proved.

\begin{pro}\label{pro2.3}
Let $Z$ be a normed space, $E$ be a closed subspace of $Z$ with  ${\rm codim}(E)<+\infty$,  and let $H$ be a subspace of $Z$. Then the following statements hold:\\
(i) $E+H+\hat P$ is a  polyhedron  in  $Z$  for each polyhedron $\hat P$  in some finite dimensional subspace of $Z$.\\
(ii) $H+P$ is a polyhedron for each polyhedron $P$ in $Z$.
\end{pro}

The following lemma is useful in the proofs of some main results.

\begin{lem}
Let  $C_1,\cdots,C_m$ be closed sets in a normed space $Z$ such that $B(x_0,r_0)\subset \bigcup\limits_{i=1}^mC_i$ for some $x_0\in Z$ and $r_0>0$. Then there exists $i_0\in\overline{1m}$ such that $B(x_0,r_0)\cap{\rm int}(C_{i_0})\not=\emptyset$.
\end{lem}

\begin{proof}
By the assumption,  $B(x_0,r_0)\setminus \bigcup\limits_{i=1}^{m-1}C_i$ is open, and  $B(x_0,r_0)\setminus \bigcup\limits_{i=1}^{m-1}C_i\subset B(x_0,r_0)\cap{\rm int}(C_m)$. Hence either $B(x_0,r_0)\cap{\rm int}(C_m)\not=\emptyset$ or  $B(x_0,r_0)\subset \bigcup\limits_{i=1}^{m-1}C_i$, which implies clearly that the conclusion holds. The proof is complete.
\end{proof}

With the help of Lemma 2.1, we can prove the following interesting proposition.

\begin{pro}\label{pro2.5}
Let $C$ be a convex set in a normed space $Z$ and let  $F_1,\cdots,F_\nu$ be exposed faces of  a polyhedron $P$ in $Z$ such that $C\subset\bigcup\limits_{j=1}^\nu F_j$. Then there exists $j_0\in\overline{1\nu}$ such that $C\subset F_{j_0}$.
\end{pro}

\begin{proof}
By Proposition \ref{pro2.1}, there exist two closed subspaces $Z_1$ and $Z_2$ of $Z$ and a polyhedron $\hat P$ in $Z_2$ such that (\ref{1}) and (\ref{2})   hold. Hence, by Corollary \ref{pro2.4}, there exists an exposed face $\hat F_j$ of $\hat P$ such that $F_j=Z_1+\hat F_j$ ($j\in\overline{1\nu}$), and so
$C\subset\bigcup\limits_{j=1}^\nu(Z_1+\hat F_j)$. It follows from (\ref{P2.1}) and (\ref{1}) that $\Pi_{Z_2}(C)\subset\bigcup\limits_{j=1}^\nu\hat F_j$. Thus, it suffices to show that $\Pi_{Z_2}(C)\subset \hat F_{j_0}$ for some $j_0\in\overline{1\nu}$. To prove this,  take $(\hat u_j^*,\alpha_j)\in Z_2^*\times\mathbb{R}$ such that
\begin{equation}\label{F2.21}
\alpha_j=\sup\limits_{x_2\in \hat P}\langle \hat u_j^*,x_2\rangle\;\;{\rm and}\;\; \hat F_j=\{x_2\in\hat P:\;\langle \hat u_j^*,x_2\rangle=\alpha_j\}\quad\forall j\in\overline{1\nu}.
\end{equation}
Since $\Pi_{Z_2}(C)$ is a convex set in the finite dimensional space $Z_2$, ${\rm rint}(\Pi_{Z_2}(C))\not=\emptyset$. Take  $\hat x\in {\rm rint}(\Pi_{Z_2}(C))$ and let $Z_3:={\rm span}(\Pi_{Z_2}(C)-\hat x)$. Then $Z_3$ is a subspace of $Z_2$ and there exists $\delta>0$ such that
$\hat x+B_{Z_3}(0,\delta)\subset \Pi_{Z_2}(C)\subset\bigcup\limits_{j=1}^\nu\hat F_j$, that is,
$$B_{Z_3}(0,\delta)\subset\bigcup\limits_{j=1}^\nu(\hat F_j-\hat x)\cap Z_3.$$
Hence, by Lemma 2.1, there exist $z_3\in B_{Z_3}(0,\delta)$, $\varepsilon\in(0,\;+\infty)$ and $j_0\in\overline{1\nu}$ such that
$z_3+B_{Z_3}(0,\varepsilon)\subset (\hat F_{j_0}-\hat x)\cap Z_3$. Letting $\hat u:=\hat x+z_3$, one has $\hat u+B_{Z_3}(0,\varepsilon)\subset \hat F_{j_0}$.
This and  (\ref{F2.21}) imply  that $\langle \hat u^*_{j_0},\hat v\rangle=0$ for all $\hat v\in B_{Z_3}(0,\varepsilon)$ and so $\langle \hat u^*_{j_0},\hat v\rangle=0$ for all $\hat v\in {Z_3}$. Hence,
$\Pi_{Z_2}(C)\subset \hat x+Z_3=\hat u+Z_3\subset\{x_2\in Z_2:\;\langle \hat u_{j_0}^*,x_2\rangle=\alpha_{j_0}\}$. Noting that $\Pi_{Z_2}(C)\subset \bigcup\limits_{j=1}^\nu \hat F_j\subset \hat P$, it follows that
$$\Pi_{Z_2}(C)\subset \hat P\cap \{x_2\in Z_2:\;\langle \hat u_{j_0}^*,x_2\rangle=\alpha_{j_0}\}=\hat F_{j_0}.$$
The proof is complete.
\end{proof}

We also need  the following proposition.

\begin{pro}\label{pro2.6}
Let $P_i$ be polyhedra in a normed space $Z$ such that  ${\rm int}(P_i)\not=\emptyset$ ($i=1,\cdots,m$).
Then  there exist polyhedra $Q_j$ in $Z$  with ${\rm int}(Q_j)\not=\emptyset$ ($j=1,\cdots,\nu$) such that
$\bigcup\limits_{i=1}^mP_i=\bigcup\limits_{j=1}^\nu Q_j$  and ${\rm int}(Q_{j})\cap Q_{j'}=\emptyset$ for all $j,j'\in\overline{1\nu}\;{\rm with}\;j\not=j'$.
\end{pro}

\begin{proof}
The conclusion holds clearly when $m=1$.
Given a natural number $n$, suppose that the conclusion holds when $m=n$. Let  $P_1,\cdots,P_n,P_{n+1}$ be arbitrary $n+1$  polyhedra in  $Z$ such that each ${\rm int}(P_i)$ is nonempty. Then, by induction, it suffices to show that there exist polyhedra $Q_j$ in $Z$  with ${\rm int}(Q_j)\not=\emptyset$ ($j=1,\cdots,\nu$) such that
$\bigcup\limits_{i=1}^{n+1}P_i=\bigcup\limits_{j=1}^\nu Q_j$ and ${\rm int}(Q_{j})\cap Q_{j'}=\emptyset$ for all $j,j'\in\overline{1\nu}$  with $j\not=j'$.
To do this, take
polyhedra $H_1,\cdots,H_l$ in $Z$ such that
\begin{equation}\label{2.21'}
\quad\bigcup\limits_{i=1}^{n}P_i=\bigcup\limits_{i=1}^l H_i,\;{\rm int}(H_i)\not=\emptyset\;\;{\rm and}\;\;H_i\cap{\rm int}(H_{i'})=\emptyset\;\;\;\forall i,i'\in\overline{1l}\;{\rm with}\;i\not=i'.
\end{equation}
If ${\rm int}(P_{n+1})\subset \bigcup\limits_{i=1}^l H_i$, then $P_{n+1}\subset \bigcup\limits_{i=1}^l H_i$ and so $\bigcup\limits_{i=1}^{n+1}P_i=\bigcup\limits_{i=1}^l H_i$; hence the conclusion is trivially true. Next suppose that  ${\rm int}(P_{n+1})\nsubseteq \bigcup\limits_{i=1}^lH_i$.
For each $i\in\overline{1l}$, take $x_{ij}^*\in Z^*\setminus\{0\}$ and $t_{ij}\in \mathbb{R}$ ($j=1,\cdots,\kappa_i$)
such that $H_i=\bigcap\limits_{j=1}^{\kappa_i}\{x\in Z:\;\langle x_{ij}^*,x\rangle\leq t_{ij}\}$. Then
$$Z\setminus H_i=\bigcup\limits_{j=1}^{\kappa_i}\{x\in Z:\;\langle x_{ij}^*,x\rangle>t_{ij}\}=\bigcup\limits_{k=1}^{\kappa_i}\Lambda^i_k\cap\{x\in Z:\;\langle x_{ik}^*,x\rangle>t_{ik}\},$$
where $\Lambda^i_1:=Z$ and  $\Lambda^i_k:=\bigcap\limits_{j=1}^{k-1}\{x\in Z:\;\langle x_{ij}^*,x\rangle\leq t_{ij}\}$ for $k=2,\cdots,\kappa_i$.
Since ${\rm int}(P_{n+1})\setminus H_i={\rm int}(P_{n+1})\cap (Z\setminus H_i)$,
\begin{equation}\label{Qi}
{\rm int}(P_{n+1})\setminus H_i=\bigcup\limits_{k=1}^{\kappa_i} {\rm int}(P_{n+1})\cap\Lambda^i_{k}\cap \{x\in Z:\;\langle x_{ik}^*,x\rangle>t_{ik}\}.
\end{equation}
For each $k\in\overline{1\kappa_i}$, let $Q_k^i:=P_{n+1}\cap\Lambda^i_{k}\cap\{x\in Z:\;\langle x_{ik}^*,x\rangle\geq t_{ik}\}$. Then each $Q_k^i$ is a polyhedron in $Z$ and, by Corollary \ref{pro2.1},
\begin{equation}\label{int}
{\rm int}(Q_k^i)={\rm int}(P_{n+1})\cap{\rm int}(\Lambda^i_{k})\cap\{z\in Z:\;\langle x_{ik}^*,z\rangle>t_{ik}\}.
\end{equation}
Since ${\rm int}(\Lambda^i_{k})=\{z\in Z:\;\langle x_{ij}^*,x\rangle< t_{ij},\; j=1,\cdots,k-1\}$,
\begin{equation}\label{K}
Q_k^i\cap {\rm int}(Q_{k'}^i)=\emptyset\quad\forall k,k'\in \overline{1\kappa_i}\;{\rm with}\;k\not=k'.
\end{equation}
Let
$$\Gamma:=\left\{(k_1,\cdots,k_l)\in \overline{1\kappa_1}\times\cdots\times \overline{1\kappa_l}:\;\bigcap\limits_{i=1}^l{\rm int}(Q_{k_i}^i)\not=\emptyset\right\}$$ and
$Q_{(k_1,\cdots,k_l)}:=\bigcap\limits_{i=1}^lQ_{k_i}^i$ for all $(k_1,\cdots,k_l)\in \Gamma$.
Then, each $Q_{(k_1,\cdots,k_l)}$ is a polyhedron in $Z$ with ${\rm int}(Q_{(k_1,\cdots,k_l)})=\bigcap\limits_{i=1}^l{\rm int}(Q_{k_i}^i)$ (thanks to Corollary 2.1). Hence, by (\ref{K}),
\begin{equation}\label{I2.12}
Q_{(k_1,\cdots,k_l)}\cap {\rm int}(Q_{(k'_1,\cdots,k'_l)})=\emptyset
\end{equation}
for all $(k_1,\cdots,k_l)\in\Gamma$ and  all $(k'_1,\cdots,k'_l)\in\Gamma\setminus\{(k_1,\cdots,k_l)\}$.
Let
$$\widetilde Q_k^i:={\rm int}(P_{n+1})\cap\Lambda^i_{k}\cap\{x\in Z:\;\langle x_{ik}^*,x\rangle>t_{ik}\}.$$
Then ${\rm int}(P_{n+1})\setminus H_i=\bigcup\limits_{k=1}^{\kappa_i}\widetilde Q_k^i$ (by (\ref{Qi})) and ${\rm cl}(\widetilde Q_k^i)\subset Q_k^i$. For any  $(k_1,\cdots,k_l)\in I_1\times\cdots\times I_l$, it is easy from (\ref{int}) to verify that $\bigcap\limits_{i=1}^l\widetilde Q_{k_i}^i\not=\emptyset$ if and only if $\bigcap\limits_{i=1}^l{\rm int}(Q_{k_i}^i)\not=\emptyset$. Noting that  $P_{n+1}\setminus H_i\subset{\rm cl}({\rm int}(P_{n+1})\setminus H_i)$, by (\ref{Qi}) and the definition of $Q_{(k_1,\cdots,k_l)}$, one has
$$P_{n+1}\setminus\bigcup\limits_{i=1}^lH_i=\bigcap\limits_{i=1}^l(P_{n+1}\setminus H_i)\subset\bigcup\limits_{(k_1,\cdots,k_l)\in\Gamma}\bigcap\limits_{i=1}^l{\rm cl}(\widetilde Q_{k_i}^i)\subset \bigcup\limits_{(k_1,\cdots,k_l)\in\Gamma}Q_{(k_1,\cdots,k_l)}\subset P_{n+1}.$$
It follows from (\ref{2.21'}) that $\bigcup\limits_{i=1}^{n+1}P_i=
\left(\bigcup\limits_{i=1}^lH_i\right)\cup\left(\bigcup\limits_{(k_1,\cdots,k_l)\in\Gamma}Q_{(k_1,\cdots,k_l)}\right)$. By  (\ref{I2.12}) and (\ref{2.21'}),  this shows that the conclusion also holds when $m=n+1$.
The proof is complete.
\end{proof}

For $(u_1^*,s_1),\cdots,(u_n^*,s_n)\in Z^*\times\mathbb{R}$ and $P=\{z\in Z:\;\langle u_i^*,z\rangle\leq s_i,\;i\in\overline{1n}\}$,
we say that $(u_i^*,s_i)$ is a redundant generator of $P$ if
$P=\big\{z\in Z:\;\langle u_j^*,z\rangle\leq s_j,\;j\in\overline{1n}\setminus\{i\}\big\}$
(cf. \cite{Te, KTZ}). For convenience, we adopt the following notion.

{\bf Definition 2.1} {\it We say that $\{(u_1^*,s_1),\cdots,(u_n^*,s_n)\}\subset Z^*\times \mathbb{R}$ is  a prime generator group of  a polyhedron $P$ in a normed space $Z$  if
\begin{equation}\label{P2.2}
P=\{z\in Z:\;\langle u_i^*,z\rangle\leq s_i,\;i\in\overline{1n}\}
\end{equation}
and $(u_i^*,s_i)$ is not a redundant generator of $P$ for all $i\in\overline{1n}$.}

Every  polyhedron has a prime generator group (cf. \cite{BCL, Te}). It is clear that if $\{(u_1^*,s_1),\cdots,(u_n^*,s_n)\}\subset Z^*\times \mathbb{R}$ is   a prime generator group of  $P$ then
\begin{equation}\label{P2.3}
P\not=\{z\in Z:\;\langle u_i^*,z\rangle\leq s_i,\;i\in\overline{1n}\setminus\{j\}\}\quad\forall j\in\overline{1n}.
\end{equation}
In the remainder of this paper, we assume that every polyhedron $P$ of $Z$ is not equal to $Z$. So, it is clear that $u_i^*\not=0$ for all $i\in\overline{1n}$
whenever $\{(u_1^*,s_1),\cdots,(u_n^*,s_n)\}$  is a prime generator group of  $P$. Moreover, we have the following lemma.

\begin{lem}\label{lemN2.1}
Let $\{(u_1^*,s_1),\cdots,(u_n^*,s_n)\}$ be a prime generator group of  a polyhedron $P$ in a normed space $Z$ and  suppose that ${\rm int}(P)\not=\emptyset$. Then $u_{j_1}^*$ and $u_{j_2}^*$ are linearly independent whenever  $j_1\in\overline{1n}$ and  $j_2\in\overline{1n}\setminus\{j_1\}$ satisfy $\langle u_{j_1}^*,\bar x\rangle=s_{j_1}$ and $\langle u_{j_2}^*,\bar x\rangle=s_{j_2}$ for some $\bar x\in P$.
\end{lem}

\begin{proof}
Suppose to the contrary that there exist $j_1\in\overline{1n}$, $j_2\in\overline{1n}\setminus\{j_1\}$, $\bar x\in P$ and $\alpha\in\mathbb{R}\setminus\{0\} $ such that
\begin{equation}\label{P2.14}
\langle u_{j_1}^*,\bar x\rangle=s_{j_1},\; \langle u_{j_2}^*,\bar x\rangle=s_{j_2}\;{\rm and}\; u_{j_2}^*=\alpha u_{j_1}^*.
\end{equation}
Take $x_0\in{\rm int}(P)$ and $r>0$ such that $B(x_0,r)\subset P$. Then $\langle u_{j_1}^*,x\rangle\leq s_{j_1}$ and $\langle u_{j_2}^*,x\rangle\leq s_{j_2}$ for all $x\in B(x_0,r)$. This and (\ref{P2.14}) imply that $\alpha>0$ and
$$\{x\in Z:\;\langle u_{j_1}^*,x\rangle \leq s_{j_1}\}=\{x\in Z:\;\langle u_{j_2}^*,x\rangle \leq s_{j_2}\}.$$
Thus,
$P=\{x\in Z:\;\langle u_i^*,x\rangle\leq s_i,\;i\in\overline{1n}\setminus\{j_{1}\}\}$, contradicting (\ref{P2.3}).
\end{proof}

\begin{lem}\label{lem2.1}
Let $\{(u_1^*,s_1),\cdots,(u_n^*,s_n)\}$ be a prime generator group of  a polyhedron $P$ in a normed space $Z$. Then, for each $j\in\overline{1n}$,
\begin{equation}\label{F1}
F_j(P):=P\cap\{x\in Z:\;\langle u_j^*,x\rangle=s_j\}\not=\emptyset.
\end{equation}
\end{lem}

Lemma \ref{lem2.1} is immediate from Definition 2.1.   The following two lemmas will be quite useful in the proof of our main result.

\begin{lem}\label{lem2.2}
Let $\{(u_1^*,s_1),\cdots,(u_n^*,s_n)\}$ be a prime generator group
of a polyhedron $P$ in a normed space $Z$. Let $F_j(P)$ be as in (\ref{F1}) and
\begin{equation}\label{F2}
F_j^\circ(P):=\{z\in Z:\;\langle u_j^*,z\rangle=s_j\;{\rm and}\;\langle u_i^*,z\rangle<s_i,\;i\in\overline{1n}\setminus\{j\}\}
\end{equation}
for all $j\in\overline{1n}$. Then the following statements are equivalent:\\
(i) ${\rm int}(P)\not=\emptyset$.\\
(ii) $F_j(P)={\rm cl}(F_j^\circ(P))$ for all $j\in\overline{1n}$.\\
(iii) $F_j^\circ(P)\not=\emptyset$ for all $j\in\overline{1n}$.\\
(iv) $F_{j_0}^\circ(P)\not=\emptyset$  for some $j_0\in\overline{1n}$.
\end{lem}

\begin{proof}
First suppose that (i) holds. Then,   by Corollary \ref{coro2.1}, there exists $x_0\in Z$ such that $\langle u_i^*,x_0\rangle<s_i$ for all $i\in\overline{1n}$. For each $j\in\overline{1n}$, by (\ref{P2.3}), there exists $v\in Z$ such that $\langle u_j^*,v\rangle>s_j$ and $\langle u_i^*,v\rangle \leq s_i$ for all $i\in\overline{1n}\setminus\{j\}$.
It follows that there exists $\lambda_0\in(0,\;1)$ such that
$$\langle u_j^*,\lambda_0x_0+(1-\lambda_0)v\rangle=s_j\;\;{\rm and}\;\;\langle u_i^*,\lambda_0x_0+(1-\lambda_0)v\rangle\rangle <s_i\quad\forall i\in\overline{1n}\setminus\{j\}.$$
Therefore,   $\frac{kx}{1+k}+\frac{\lambda_0x_0+(1-\lambda_0)v}{k+1}\in F_j^\circ(P)$ for all $(x,k)\in F_j(P)\times\mathbb{N}$. Letting $k\rightarrow\infty$, it follows that $x\in{\rm cl}(F_j^\circ(P))$ for all $x\in F_j(P)$, that is,  $F_j(P)\subset {\rm cl}(F_j^\circ(P))$. Since the converse inclusion holds trivially, this shows implication (i)$\Rightarrow$(ii). Since (ii)$\Rightarrow$(iii) is immediate from Lemma \ref{lem2.1} and (iii)$\Rightarrow$(iv) is trivial,  it suffices to show (iv)$\Rightarrow$(i). To prove this, let $\bar x\in F_{j_0}^\circ(P)$, that is, $\langle u_{j_0}^*,\bar x\rangle=s_{j_0}$ and $\langle u_i^*,\bar x\rangle<s_i$ for all $i\in\overline{1n}\setminus\{j_0\}$. Taking $h\in Z$ with $\langle u_{j_0}^*,h\rangle<0$ (thanks to  $u_{j_0}^*\not=0$), it follows that there exists  $t>0$ sufficiently small
such that $\langle u_k^*,\bar x+th\rangle<s_k$ for all $k\in\overline{1n}$. This shows  that $\bar x+th\in{\rm int}(P)$, and hence (iv)$\Rightarrow$(i) holds. The proof is complete.
\end{proof}

\begin{lem}\label{lem2.3}
Let $P_1$ and $P_2$ be two polyhedra in a normed space $Z$ such that ${\rm int}(P_1)\cap P_2=\emptyset$, and let
$\{(u_{ij}^*,s_{ij}):\;j=1,\cdots,n_i\}\subset Z^*\times\mathbb{R}$ be a prime generator group of $P_i$ ($i=1,2$).  Then for any $(j_1,j_2)\in\overline{1n_1}\times\overline{1n_2}$ and  $x_0\in F_{j_1}^\circ(P_1)\cap F_{j_2}^\circ(P_2)$ there exists $r>0$ such that
$\mathcal{N}(u_{1j_1}^*)=\mathcal{N}(u_{2j_2}^*)$  and
\begin{equation}\label{new2.14}
\;\;F_{j_1}^\circ(P_1)\cap B_Z(x_0,r)=F_{j_2}^\circ(P_2)\cap B_Z(x_0,r)=(x_0+\mathcal{N}(u_{1j_1}^*))\cap B_Z(x_0,r),
\end{equation}
where  $B_Z(x_0,r):=\{x\in Z:\;\|x-x_0\|<r\}$ and $F_{j_1}^\circ(P_1)$ is as in (\ref{F2}).
\end{lem}

\begin{proof}
Let $(j_1,j_2)\in\overline{1n_1}\times\overline{1n_2}$ and  $x_0\in F_{j_1}^\circ(P_1)\cap F_{j_2}^\circ(P_2)$. Then $x_0\in P_1\cap P_2$.  Since ${\rm int}(P_1)\cap P_2=\emptyset$, the separation theorem implies that there exists $v^*\in Z^*\setminus\{0\}$ such that
$\langle v^*,x_0\rangle=\inf\limits_{x\in P_1}\langle v^*,x\rangle=\sup\limits_{x\in P_2}\langle v^*,x\rangle$.
Noting that
\begin{equation}\label{F3}
F_{j_1}^\circ(P_1)\cap B_Z(x_0,r)=(x_0+\mathcal{N}(u_{1j_1}^*))\cap B_Z(x_0,r)\subset P_1
\end{equation}
and
\begin{equation}\label{F4}
F_{j_2}^\circ(P_2)\cap B_Z(x_0,r)=(x_0+\mathcal{N}(u_{2j_2}^*))\cap B_Z(x_0,r)\subset P_2
\end{equation}
for some $r>0$  (thanks to the definitions of $F_{j_1}^\circ(P_1)$ and $F_{j_2}^\circ(P_2)$), it follows that
$$\langle v^*,x_0\rangle=\inf\limits_{x\in (x_0+\mathcal{N}(u_{1j_1}^*))\cap B_Z(x_0,r)}\langle v^*,x\rangle=\sup\limits_{x\in (x_0+\mathcal{N}(u_{2j_2}^*))\cap B_Z(x_0,r)}\langle v^*,x\rangle.$$
Hence $\inf\limits_{x\in \mathcal{N}(u_{1j_1}^*)\cap B_Z(0,r)}\langle v^*,x\rangle=\sup\limits_{x\in \mathcal{N}(u_{2j_2}^*)\cap B_Z(0,r)}\langle  v^*,x\rangle=0$, and so $$\mathcal{N}(v^*)=\mathcal{N}(u_{1j_1}^*)=\mathcal{N}(u_{2j_2}^*)$$
because $v^*$ is linear and both $\mathcal{N}(u_{1j_1}^*)$ and $\mathcal{N}(u_{2j_2}^*)$ are maximal linear subspaces of $Z$. This, together with (\ref{F3}) and (\ref{F4}), implies that (\ref{new2.14}) holds.
\end{proof}

\section{Piecewise linear vector-valued functions}
In this section, we will distinguish $\mathcal{PL}_1(X,Y)$ and $\mathcal{PL}(X,Y)$ and consider the structure of a  piecewise linear function.

\begin{pro}\label{pro3.1}
Let $X$ and $Y$ be normed spaces. Then the following statements hold.\\
(i) $\mathcal{L}(X,Y)$ is always contained in $\mathcal{P}\mathcal{L}(X,Y)$.\\
(ii) $\mathcal{P}\mathcal{L}_1(X,Y)\not=\emptyset$ if and only if  ${\rm dim}(Y)<\infty$. Consequently, if $Y$ is infinite dimensional, then every linear function from $X$ to $Y$ is not in $\mathcal{P}\mathcal{L}_1(X,Y)$.\\
(iii) $\mathcal{P}\mathcal{L}_1(X,Y)=\mathcal{P}\mathcal{L}(X,Y)$ when ${\rm dim}(Y)<\infty$.
\end{pro}

\begin{proof}
Since (i) is trivial and the sufficiency part of (ii) is a straightforward consequence of (i) and (iii), it suffices to show (iii) and the necessity part of (ii). First suppose that $\mathcal{P}\mathcal{L}_1(X,Y)\not=\emptyset$, and let
 $g$ be an  element in $\mathcal{P}\mathcal{L}_1(X,Y)$. Then there exist finitely many polyhedra $\Lambda_1,\cdots,\Lambda_k$ in the product $X\times Y$ such that
\begin{equation}\label{4.2}
{\rm gph}(g)=\bigcup\limits_{i=1}^k\Lambda_i\;\;{\rm and}\;\;X=\bigcup\limits_{i=1}^k\Lambda_i|_{X},
\end{equation}
where $\Lambda_i|_X:=\{x\in X:\;(x,y)\in\Lambda_i\;{\rm for\;some}\;y\in Y\}$ is the projection of $\Lambda_i$ to $X$.
Given an  $i\in\overline{1k}$, by Proposition \ref{pro2.1}, there exist two closed subspaces $X_i,\tilde{X}_i$ of $X$ and two closed subspaces $Y_i,\tilde{Y}_i$ of $Y$ such that
\begin{equation}\label{n3.2}
X\times Y=(X_i\times Y_i)\oplus(\tilde{X}_i\times\tilde{Y}_i), \;\;{\rm codim}(X_i\times Y_i)={\rm dim}(\tilde{X}_i\times\tilde{Y}_i)\leq\infty,
\end{equation}
\begin{equation}\label{n3.3}
\Lambda_i=X_i\times Y_i+\tilde\Lambda_i,
\end{equation}
where $\tilde\Lambda_i$ is a polyhedron in $\tilde{X}_i\times \tilde{Y}_i$. Since $g$ is a single-valued function, it follows from (\ref{4.2}) that $Y_i=\{0\}$ and $\tilde{Y}_i=Y$. Hence $Y$ is finite-dimensional, and the necessity part of (ii) is proved.  Next we prove $g\in\mathcal{P}\mathcal{L}(X,Y)$. To prove this,  we only need to  show that  there exist $T_i\in\mathcal{L}(X,Y)$ and $b_i\in Y$ such that
\begin{equation}\label{n3.1}
g(x)=T_i(x)+b_i\quad\forall x\in \Lambda_i|_{X}.
\end{equation}
Since every convex set in a finite-dimensional space has a nonempty relative interior, ${\rm rint}(\tilde\Lambda_i)\not=\emptyset$. Take a point $(\tilde{a}_i,\tilde{b}_i)$ in ${\rm rint}(\tilde\Lambda_i)$. Thus, $\tilde a_i\in{\rm rint}(\tilde\Lambda_i|_{\tilde X_i})$, and $E_i:=\mathbb{R}_+(\tilde\Lambda_i|_{\tilde X_i}-\tilde{a}_i)$ and $Z_i:=\mathbb{R}_+(\tilde\Lambda_i-(\tilde{a}_i,\tilde{b}_i))$ are  linear subspaces of $\tilde{X}_i$ and $\tilde{X}_i\times\tilde{Y}_i$, respectively.  Noting that $\tilde\Lambda_i\subset {\rm gph}(g)$, define $\hat T_i:\tilde\Lambda_i|_{\tilde X_i}-\tilde{a}_i\rightarrow \tilde Y_i$ such that
\begin{equation}\label{4.6}
\hat T_i(u_i):=g(u_i+\tilde{a}_i)-g(\tilde{a}_i)=g(u_i+\tilde{a}_i)-\tilde{b}_i\quad\forall u_i\in \tilde\Lambda_i|_{\tilde X_i}-\tilde{a}_i.
\end{equation}
Then  ${\rm gph}(\hat T_i)=\tilde\Lambda_i-(\tilde{a}_i,\tilde{b}_i)$.  Let
$\tilde T_i:E_i\rightarrow\tilde Y_i$ be such that
$$\tilde T_i(tu_i):=t\hat T_i(u_i)\quad\forall (t,u_i)\in \mathbb{R}_+\times(\tilde\Lambda_i|_{\tilde X_i}-\tilde{a}_i).$$
It is easy to verify that $\tilde{T}_i$ is well-defined and its graph  is just the linear subspace $Z_i=\mathbb{R}_+(\tilde\Lambda_i-(\tilde{a}_i,\tilde{b}_i))$, and so $\tilde T_i$ is linear. Hence there exist $e_j\in Y$ and $e_{ij}^*\in E_i^*$ ($j=1,\cdots,p$) such that  $e_1,\cdots,e_p$ are linearly independent and
$$\tilde T_i(x)=\sum\limits_{j=1}^p\langle e_{ij}^*,x\rangle e_j\quad\forall x\in E_i.$$
For each $j\in\overline{1p}$, let $\tilde{e}_{ij}^*:X_i+E_i\rightarrow\mathbb{R}$ be such that
$$\langle\tilde{e}_{ij}^*,u+v\rangle=\langle e_{ij}^*,v\rangle\quad\forall (u,v)\in X_i\times E_i.$$
Then, by (\ref{n3.2}) and $E_i\subset\tilde{X}_i$, $\tilde{e}_{ij}^*$ is a linear functional on $X_i+E_i$, and its null space
$$\mathcal{N}(\tilde{e}_{ij}^*):=\{x\in X_i+E_i:\;\langle\tilde{e}_{ij}^*,x\rangle=0\}=X_i+\{v\in E_i:\;\langle e_{ij}^*,v\rangle=0\}.$$
Since $X_i$ is a closed subspace of $X$ and ${\rm dim}(E_i)<\infty$, it follows that $\mathcal{N}(\tilde{e}_{ij}^*)$ is a closed subspace of $X$. Hence $\tilde{e}_{ij}^*$ is a continuous linear functional on $X_i+E_i$ (thanks to \cite[Theorem 1.18]{Ru}). By the Hahn-Banach theorem, there exists $x_{ij}^*\in X^*$ such that $x_{ij}^*|_{X_i+E_i}=\tilde{e}_{ij}^*$. Let $T_i:X\rightarrow Y$ be such that
$$T_i(x)=\sum\limits_{j=1}^p\langle x_{ij}^*,x\rangle e_j\quad\forall x\in X.$$
Then $T_i\in\mathcal{L}(X,Y)$,
\begin{equation}\label{4.7}
\mathcal{N}(T_i)\supset\bigcap\limits_{j=1}^p\mathcal{N}(x_{ij}^*)\supset X_i\;\;{\rm and}\;\;T_i|_{\tilde{\Lambda}_i|_{\tilde{X}_i}-\tilde{a}_i}=\tilde{T}_i|_{\tilde{\Lambda}_i|_{\tilde{X}_i}-\tilde{a}_i}=\hat{T}_i.
\end{equation}
Let $x$ be an arbitrary element in $\Lambda_i|_X$ and take $y\in Y$ such that $(x,y)\in\Lambda_i$. Then, by (\ref{n3.2}) and (\ref{n3.3}), there exist $x_i\in X_i$ and  $\tilde{x}_i\in\tilde{\Lambda}_i|_{\tilde{X}_i}$ such that $(\tilde{x}_i,y)\in\tilde{\Lambda}_i$ and $(x,y)=(x_i+\tilde{x}_i,y)$ (because  $Y_i=\{0\}$). Hence, by (\ref{4.6}) and (\ref{4.7}), one has
$$g(x)=g(\tilde{x}_i)=y=\hat{T}_i(\tilde{x}_i-\tilde{a}_i)+\tilde{b}_i=T_i(\tilde{x}_i-\tilde{a}_i)+\tilde{b}_i=T_i(x)-T_i(\tilde{a}_i)+\tilde{b}_i.$$
This shows that (\ref{n3.1}) holds with $b_i=-T_i(\tilde{a}_i)+\tilde{b}_i$ and so $g\in \mathcal{P}\mathcal{L}(X,Y)$. Therefore,
$\mathcal{P}\mathcal{L}_1(X,Y)\subset \mathcal{P}\mathcal{L}(X,Y)$.

Now suppose that ${\rm dim}(Y)<\infty$. To prove the converse inclusion
$\mathcal{P}\mathcal{L}_1(X,Y)\supset \mathcal{P}\mathcal{L}(X,Y)$,
let $g\in\mathcal{P}\mathcal{L}(X,Y)$. Then there exist  $P_i\in\mathcal{P}(X)$, $T_i\in\mathcal{L}(X,Y)$ and $b_i\in Y$ ($i=1,\cdots,n$) such that
\begin{equation}\label{n3.5}
X=\bigcup\limits_{i=1}^nP_i\;\;{\rm and}\;\;g(x)=T_i(x)+b_i\quad\forall x\in P_i\;{\rm and}\;\forall i\in\overline{1n}.
\end{equation}
By ${\rm dim}(Y)<\infty$, there exist $y_1^*,\cdots,y_q^*\in Y^*$ such that $Y^*={\rm span}\{y_1^*,\cdots,y_q^*\}$.  For any $x\in X$, since
$$T_i(x)=y\Leftrightarrow[\langle y^*,T_i(x)\rangle=\langle y^*,y\rangle\;\;\forall y^*\in Y^*]\Leftrightarrow[\langle y_j^*,T_i(x)\rangle=\langle y_j^*,y\rangle,\;j=1,\cdots,q],$$
$$T_i(x)=y\Longleftrightarrow[\langle T_i^*(y_j^*),x\rangle=\langle y_j^*,y\rangle,\;j=1,\cdots,q].$$
Hence ${\rm gph}(T_i)=\{(x,y)\in X\times Y:\;\langle T_i^*(y_j^*),x\rangle-\langle y_j^*,y\rangle=0,\;j=1,\cdots,q\}$,
and so ${\rm gph}(T_i)$ is a polyhedron of $X\times Y$. Noting (by (\ref{n3.5})) that
$${\rm gph}(g)=\bigcup\limits_{i=1}^n({\rm gph}(T_i)+(0,b_i))\cap(P_i\times Y),$$
it follows that ${\rm gph}(g)$  is the union of finitely many polyhedra in $X\times Y$. Therefore, $g\in\mathcal{P}\mathcal{L}_1(X,Y)$. The proof of (iii) is complete.
\end{proof}

Given $f\in\mathcal{PL}(X,Y)$,  there exist $(P_1,T_1,b_1),\cdots, (P_m,T_m,b_m)$ in the product $\mathcal{P}(X)\times\mathcal{L}(X,Y)\times Y$ such that (\ref{pl2}) holds. For  $i\in\overline{1m}$,  since each polyhedron is closed, the first equality of (\ref{pl2}) implies that ${\rm int}(P_i)\supset X\setminus\bigcup\limits_{j\in\overline{1m}\setminus\{i\}}P_j$ and so $X=\bigcup\limits_{j\in\overline{1m}\setminus\{i\}}P_j$ whenever ${\rm int}(P_i)=\emptyset$. Hence, without loss of generality, we can assume that each $P_i$ in (\ref{pl2}) has a nonempty interior. Moreover,
we assume without loss of generality that there exists $k\in\overline{1m}$ satisfying the following property:
\begin{equation}\label{3.9}
(T_i,b_i)\not=(T_{i'},b_{i'})\quad\forall i,i'\in \overline{1k}\;{\rm with}\; i\not=i'
\end{equation}
and for each $j\in \overline{1m}$ there exists $i\in \overline{1k}$ such that $(T_j,b_j)=(T_i,b_i)$.
For each $i\in\overline{1k}$, let
\begin{equation}\label{3.10}
I_i:=\{j\in \overline{1m}: (T_j,b_j)=(T_i,b_i)\}\;\;{\rm and}\;\; Q_i:=\bigcup\limits_{j\in I_i}P_j.
\end{equation}
Then $X=\bigcup\limits_{i\in\overline{1k}}Q_i$, $X\not=\bigcup\limits_{i\in\overline{1k},i\not=j}Q_i$  and $f|_{Q_j}=T_j|_{Q_j}+b_j$ for all $j\in\overline{1k}$.
We claim that
\begin{equation}\label{3.12}
{\rm int}(Q_i)\cap{\rm int}(Q_{i'})=\emptyset\quad\forall i,i'\in\overline{1k}\;{\rm with}\;i\not=i'.
\end{equation}
Indeed, if this is not the case, there exist $i,i'\in\overline{1k}$  with $i\not=i'$, $x\in X$ and $r>0$ such that $B(x,r)\subset Q_i\cap Q_{i'}$, and so
$$f(x)=T_i(u)+b_i=T_{i'}(u)+b_{i'}\quad\forall u\in B(x,r).$$Since $T_i$ and $T_{i'}$ are linear, it follows that $(T_i,b_i)=(T_{i'},b_{i'})$, contradicting (\ref{3.9}). Hence (\ref{3.12}) holds. Since each $Q_i$ is closed, (\ref{3.12}) can be rewritten as
$$Q_i\cap{\rm int}(Q_{i'})=\emptyset\quad\forall i,i'\in\overline{1k}\;{\rm with}\;i\not=i'.$$
Therefore, by Proposition \ref{pro2.6}, we have the following result.

\begin{pro}\label{pro3.2}
For each $f\in\mathcal{PL}(X,Y)$ there exist $(P_i,T_i,b_i)\in \mathcal{P}(X)\times\mathcal{L}(X,Y)\times Y$ ($i=1,\cdots,m$) such that
\begin{equation}\label{3.14}
X=\bigcup\limits_{i=1}^m P_i,\; {\rm int}(P_i)\not=\emptyset,\; P_i\cap {\rm int}(P_j)=\emptyset\;\; \forall i,j\in\overline{1m}\;{\rm with}\; i\not=j,
\end{equation}
and
\begin{equation}\label{L3.17}
f|_{P_i}=T_i|_{P_i}+b_i\quad\forall i\in\overline{1m},
\end{equation}
that is, $f(x)=T_ix+b_i$ for all $i\in\overline{1m}$ and all $x\in P_i$.
\end{pro}

Let $\{\Lambda_1,\cdots,\Lambda_m,\Lambda'_1,\cdots,\Lambda'_{m'}\}\subset\mathcal{P}(X)$ be such that $X=\bigcup\limits_{i=1}^m\Lambda_i=\bigcup\limits_{j=1}^{m'}\Lambda'_j$. Then $X=\bigcup\limits_{(i,j)\in\overline{1m}\times\overline{1m'}}\Lambda_i\cap\Lambda'_j$. Setting $I=\{(i,j)\in\overline{1m}\times\overline{1m'}:\;{\rm int}(\Lambda_i\cap\Lambda'_j)\not=\emptyset\}$, one has $X=\bigcup\limits_{(i,j)\in I}\Lambda_i\cap\Lambda'_j$. This yields the following proposition.

\begin{pro}\label{pro3.3}
For any $f,f'\in\mathcal{PL}(X,Y)$ there exist $(P_i,T_i,b_i),(P'_i,T'_i,b'_i)\in \mathcal{P}(X)\times\mathcal{L}(X,Y)\times Y$ ($i=1,\cdots,m$) such that (\ref{3.14}) holds and
\begin{equation}\label{L3.13}
f|_{P_i}=T_i|_{P_i}+b_i\;\;{\rm and}\;\;f'|_{P_i}=T_i'|_{P_i}+b'_i\quad\forall  i\in\overline{1m}.
\end{equation}
\end{pro}

Now we are ready to establish the main result in this section, which shows that any piecewise linear function defined on an infinite dimensional space $X$ can be decomposed into the sum of a linear function on an infinite dimensional closed subspace of $X$ and a piecewise linear function on a finite dimensional subspace of $X$.

\begin{them}\label{thm3.0}
Let $(P_1,T_1,b_1),\cdots,(P_m,T_m,b_m)\in \mathcal{P}(X)\times\mathcal{L}(X,Y)\times Y$  and  $f\in\mathcal{PL}(X,Y)$ be such that (\ref{3.14}) and (\ref{L3.17}) hold. Then $X_f:={\rm lin}\left(\bigcap\limits_{i=1}^m P_i\right)$ is a closed subspace of $X$ with ${\rm codim}\left(X_f\right)<\infty$ and
$$T_1(x)=\cdots=T_m(x)\quad\forall x\in X_f.$$
\end{them}

\begin{proof}
For each $i\in\overline{1m}$, take a prime generator group $\{(x_{ij}^*,t_{ij}):\;j\in\overline{1\nu_i}\}$  of $P_i$. Then,
\begin{equation}\label{3.17}
P_i=\{x\in X:\;\langle x_{ij}^*,x\rangle\leq t_{ij},\;j\in\overline{1\nu_i}\}
\end{equation}
and
\begin{equation}\label{3.17'}
P_i\not=\{x\in X:\;\langle x_{ij}^*,x\rangle\leq t_{ij},\;j\in\overline{1\nu_i}\setminus\{j'\}\}\quad\forall j'\in\overline{1\nu_i}.
\end{equation}
Hence, by (\ref{L2.2}), $X_1:=X_f=
\bigcap\limits_{i\in\overline{1m}}\bigcap\limits_{j\in\overline{1\nu_i}}\mathcal{N}(x_{ij}^*)$
is a closed subspace of $X$ with ${\rm codim}(X_1)\leq\sum\limits_{i=1}^m\nu_i$,  and so there exists a closed subspace $X_2$ of $X$ such that
\begin{equation}\label{3.19}
X=X_1\oplus X_2\;\;{\rm and}\;\; {\rm codim}(X_1)={\rm dim}(X_2)<\infty.
\end{equation}
Let
\begin{equation}\label{3.20P}
\hat P_i:=\{x\in X_2:\;\langle x_{ij}^*,x\rangle\leq t_{ij},\; \;j\in\overline{1\nu_i}\}.
\end{equation}
Then, by (\ref{3.17}) and Proposition \ref{pro2.1},
\begin{equation}\label{P3.17}
P_i=X_1+\hat P_i\;\;\;{\rm and}\;\;\;{\rm int}(P_i)=X_1+{\rm int}_{X_2}(\hat P_i).
\end{equation}
It follows from (\ref{3.14}) and  (\ref{3.19}) that
\begin{equation}\label{3.16}
{\rm int}_{X_2}(\hat P_i)\not=\emptyset\;\; {\rm and}\;\;\hat P_i\cap {\rm int}_{X_2}(\hat P_{i'})=\emptyset\;\;\forall i,i'\in\overline{1m}\;{\rm with}\;i\not=i'.
\end{equation}
Fix two arbitrary distinct indices $i$ and $i'$ in $\overline{1m}$. We only need to prove $T_i|_{X_1}=T_{i'}|_{X_1}$.  To prove this, let
$$F_{j}^\circ(P_i):=\left\{x\in X:\;\langle x_{ij}^*,x\rangle=t_{ij}\;{\rm and}\;\langle x_{il}^*,x\rangle<t_{il}\;{\rm for\;all}\;l\in\overline{1\nu_i}\setminus\{j\}\right\}$$
and
\begin{equation}\label{circ}
F_{j}^\circ(\hat P_i):=\{x\in X_2:\;\langle x_{ij}^*,x\rangle=t_{ij}\;{\rm and}\;\langle x_{il}^*,x\rangle<t_{il}\;{\rm for\;all}\;l\in\overline{1\nu_i}\setminus\{j\}\}.
\end{equation}
Then, by the definition of $X_1$ and (\ref{3.19}), $F_{j}^\circ(P_i)=X_1+F_{j}^\circ(\hat P_i)$ and  $F_{j}^\circ(\hat P_i)\not=\emptyset$ (thanks to Lemma \ref{lem2.2}). Take $(\bar u,\bar u')\in{\rm int}_{X_2}(\hat P_i)\times{\rm int}_{X_2}(\hat P_{i'})$ and $u^*\in X_2^*\setminus\{0\}$ such that $\langle u^*,\bar u'-\bar u\rangle\not=0$. Then there exists $\delta>0$ such that
\begin{equation}\label{3.23n}
\bar u+B_{X_3}(0,\delta)\subset{\rm int}_{X_2}(\hat P_i)\;\;{\rm and}\;\; \bar u'+B_{X_3}(0,\delta)\subset{\rm int}_{X_2}(\hat P_{i'}),
\end{equation}
where $X_3:=\mathcal{N}(u^*)=\{x\in X_2:\;\langle u^*,x\rangle=0\}$.
Hence
\begin{equation}\label{3.23}
{\rm dim}(X_3)={\rm dim}(X_2)-1,\;\;X_2=X_3\oplus\mathbb{R}(\bar u'-\bar u)
\end{equation}
and
\begin{equation}\label{3.24}
{\rm int}_{X_2}\big([\bar u,\;\bar u']+B_{X_3}(0,\delta)\big)=(\bar u,\;\bar u)+B_{X_3}(0,\delta)\not=\emptyset,
\end{equation}
where $[\bar u,\;\bar u']:=\{\bar u+t(\bar u'-\bar u):\;0\leq t\leq 1\}$ and $(\bar u,\;\bar u'):=\{\bar u+t(\bar u'-\bar u):\;0< t<1\}$. For each $z\in B_{X_3}(0,\delta)$, let
$$I_z:=\{i\in\overline{1m}:\;\hat P_i\cap (z+[\bar u, \;\bar u'])\;{\rm contains\;at\;least\;two\;points}\}$$
and
$$I_z^\circ:=\{i\in\overline{1m}:\;{\rm int}_{X_2}(\hat P_i)\cap (z+[\bar u, \;\bar u'])\not=\emptyset\}.$$
Then  $I_z^\circ\subset I_z$, and $\hat P_i\cap (z+[\bar u, \;\bar u'])$ contains at most  an element for all $i\in\overline{1m}\setminus I_z$. Noting that $X_2=\bigcup\limits_{i\in\overline{1m}}\hat P_i$ (thanks to (\ref{3.19})---(\ref{P3.17}) and (\ref{3.14})), it follows that
\begin{equation}\label{3.25}
z+[\bar u,\;\bar u']=\bigcup\limits_{i\in I_z}\hat P_i\cap (z+[\bar u, \;\bar u'])\quad\forall z\in B_{X_3}(0,\delta).
\end{equation}
Regarding $X_2$ as the Euclidean space $\mathbb{R}^{{\rm dim}(X_2)}$ (without loss of generality), let $\mu_{X_2}$ and $\mu_{X_3}$ denote the Lebesgue measures on $X_2$ and $X_3$, respectively. Setting
$$E_0:=\{z\in B_{X_3}(0,\delta):\;I_z^\circ\not=I_z\},$$
we claim that
$\mu_{X_3}(E_0)=0$.
To prove this, let $z$ be an arbitrary element in $E_0$. Then there exists $i_z\in I_z$ such that $i_z\not\in I^\circ_z$, and so
$\hat P_{i_z}\cap (z+[\bar u, \;\bar u'])\subset\hat P_{i_z}\setminus{\rm int}_{X_2}(\hat P_{i_z})$. Since $\hat P_{i_z}\setminus{\rm int}_{X_2}(\hat P_{i_z})$ is the union of finitely many faces of $\hat P_{i_z}$,
it follows from Proposition \ref{pro2.5} that there exists a face of $\hat P_{i_z}$ containing the convex set $\hat P_{i_z}\cap (z+[\bar u, \;\bar u'])$, that is, there exists $v^*\in X_2^*\setminus\{0\}$ such that
$$\hat P_{i_z}\cap (z+[\bar u, \;\bar u'])\subset\{x_2\in \hat P_{i_z}:\;\langle v^*,z_2\rangle=\sup\limits_{x\in \hat P_{i_z}}\langle v^*,x\rangle\}=\hat P_{i_z}\cap(v+\mathcal{N}(v^*))$$
for some $v\in X_2$ with $\langle v^*,v\rangle=\sup\limits_{x\in \hat P_{i_z}}\langle v^*,x\rangle$. Since $\hat P_{i_z}\cap (z+[\bar u, \;\bar u'])$ is a segment containing at least two points (thanks to the definition of $I_z$) and $v^*$ is linear, the entire segment $z+[\bar u, \;\bar u']$ is contained in the hyperplane $v+\mathcal{N}(v^*)$.  Noting that each $\hat P_i$ (as a polyhedron in $X_2$) has finitely many faces, it follows that  there exist $v_1^*,\cdots,v_q^*\in X_2^*\setminus\{0\}$ and $v_1,\cdots,v_q\in X_2$ such that $z+[\bar u,\;\bar u']\subset \bigcup\limits_{k=1}^q(v_k+\mathcal{N}(v_k^*))$ for all $z\in E_0$,
that is, $E_0+[\bar u,\;\bar u']\subset \bigcup\limits_{k=1}^q(v_k+\mathcal{N}(v_k^*))$.
Since each $\mathcal{N}(v_k^*)$ is of dimension ${\rm dim}(X_2)-1$,
$\mu_{X_2}(E_0+[\bar u,\;\bar u'])\leq\mu_{X_2}\left(\bigcup\limits_{k=1}^q(v_k+\mathcal{N}(v_k^*))\right)\leq \sum\limits_{k=1}^q\mu_{X_2}(v_k+\mathcal{N}(v_k^*))=0$.
Thus, by   (\ref{3.23}) and the Fubini Theorem on measure, one has $\mu_{X_3}(E_0)=0$.
Next, let
$$z\in B_{X_3}(0,\delta)\setminus E_0.$$
Then $I_{z}=I_{z}^\circ$.
Thus, by (\ref{3.25}) and the definition of $I_{z}^\circ$,
$$z+[\bar u,\;\bar u']=\bigcup\limits_{\kappa\in I_{z}^\circ}\hat P_\kappa\cap (z+[\bar u, \;\bar u'])\;\;{\rm and}\;\;{\rm int}_{X_2}(\hat P_\kappa)\cap ( z+[\bar u, \;\bar u'])\not=\emptyset\;\;\;\forall  \kappa\in I_{z}^\circ.$$
Hence, by (\ref{3.16}),  ${\rm int}_{X_2}(\hat P_\kappa)\cap ( z+[\bar u, \;\bar u'])$ and ${\rm int}_{X_2}(\hat P_{\kappa'})\cap ( z+[\bar u, \;\bar u'])$ are two disjoint open segments in $z+[\bar u,\;\bar u']$   for any $\kappa,\kappa'\in I^\circ_z$ with $\kappa\not=\kappa'$. Noting that  $z+\bar u\in{\rm int}_{X_2}(\hat P_i)$ and $z+\bar u'\in{\rm int}_{X_2}(\hat P_{i'})$ (by (\ref{3.23n})),
it follows  that there exist  $\iota_0^z,\iota_1^z,\cdots,\iota_{\gamma_z}^z\in\overline{1m}$ and $\lambda_0^z,\lambda_1^z,\cdots,\lambda_{\gamma_z}^z\in[0,\;1)$ such that
\begin{equation}\label{3.28}
I_z=I^\circ_z=\{\iota_0^z,\iota_1^z,\cdots,\iota_{\gamma_z}^z\},\; \iota_0^z=i,\;\iota_{\gamma_z}^z=i',\;\lambda_0^z=0,\;\lambda_{k-1}^z<\lambda_{k}^z,
\end{equation}
$$
z+\bar u+[0,\;\lambda_1^z)(\bar u'-\bar u)= (z+[\bar u,\;\bar u'])\cap{\rm int}_{X_2}(\hat P_i),
$$
$$
z+\bar u+(\lambda_{\gamma_z}^z,\;1](\bar u'-\bar u)= (z+[\bar u,\;\bar u'])\cap{\rm int}_{X_2}(\hat P_{i'}),
$$
$$
z+\bar u+[\lambda_{k-1}^z,\;\lambda_{k}^z](\bar u'-\bar u)=(z+[\bar u,\;\bar u'])\cap \hat P_{\iota_{k-1}^z}
$$
and
$$
z+\bar u+(\lambda_{k-1}^z,\;\lambda_{k}^z)(\bar u'-\bar u)= (z+[\bar u,\;\bar u'])\cap{\rm int}_{X_2}(\hat P_{\iota_{k-1}^z})
$$
for all $k\in\overline{1{\gamma_z}}$. Therefore
\begin{equation}\label{3.31z}
z+\bar u+\lambda_{k}^z(\bar u'-\bar u)\in \hat P_{\iota_{k-1}^z}\cap \hat P_{\iota_{k}^z}\quad\forall k\in\overline{1{\gamma_z}}.
\end{equation}
This and (\ref{3.16}) imply that $z+\bar u+\lambda_{k}^z(\bar u'-\bar u)\not\in {\rm int}_{X_2}(\hat P_{\iota_{k-1}^z})\cup {\rm int}_{X_2}(\hat P_{\iota_{k}^z})$ for all $k\in\overline{1{\gamma_z}}$.
Letting
\begin{equation}\label{Jz1}
J^-_{(z,k)}:=\{j\in  \overline{1\nu_{\iota_{k-1}^{z}}}:\;\langle x_{\iota_{k-1}^{z}j}^*,z+\bar u+\lambda_{k}^z(\bar u'-\bar u)\rangle=t_{\iota_{k-1}^{z}j}\}
\end{equation}
and
\begin{equation}\label{Jz2}
J_{(z,k)}:=\{j\in \overline{1\nu_{\iota_k^{z}}}:\;\langle x^*_{\iota_{k}^{z}j},z+\bar u+\lambda_{k}^z(\bar u'-\bar u)\rangle=t_{\iota_{k}^{z}j}\},
\end{equation}
it follows from (\ref{3.20P}) and Corollary \ref{pro2.1} that $J^-_{(z,k)}\not=\emptyset$ and $J_{(z,k)}\not=\emptyset$ for all $k\in\overline{1{\gamma_z}}$.  We claim that there exist $\bar z\in B_{X_3}(0,\delta)\setminus E_0$  and $(j_k^-,j_k)\in \overline{1\nu_{\iota_{k-1}^{\bar z}}}\times\overline{1\nu_{\iota_k^{\bar z}}}$ such that
\begin{equation}\label{Jz3}
J^-_{(\bar z,k)}=\{j_k^-\}\;\;{\rm  and}\;\; J_{(\bar z,k)}=\{j_k\}\quad\forall k\in\overline{1\gamma_{\bar z}}.
\end{equation}
Indeed, if this is not the case, for each $z\in B_{X_3}(0,\delta)\setminus E_0$ there exists $k\in\overline{1\gamma_z}$ such that either $J^-_{(z,k)}$ or $J_{(z,k)}$ contains at least two elements; we assume without loss of generality that there exist $k\in\overline{1\gamma_z}$ and $j_1,j_2\in J_{(z,k)}$ such that $j_1\not=j_2$. Then, by (\ref{3.31z}) and (\ref{Jz2}),
\begin{equation}\label{3.36z}
\quad z+\bar u+\lambda_{k}^z(\bar u'-\bar u)\in\big\{x\in \hat P_{\iota_{k}^{z}}:\;\langle x^*_{\iota_{k}^{z}j_1},x\rangle=t_{\iota_{k}^{z}j_1}\;\;{\rm and}\;\;\langle x^*_{\iota_{k}^{z}j_2},x\rangle=t_{\iota_{k}^{z}j_2}\big\}.
\end{equation}
Since $\big\{(x_{\iota_{k}^{z}1}^*,t_{\iota_{k}^{z}1}),(x_{\iota_{k}^{z}2}^*,t_{\iota_{k}^{z}2}),\cdots,
(x_{\iota_{k}^{z}\nu_{\iota_{k}^{z}}}^*,t_{\iota_{k}^{z}\nu_{\iota_{k}^{z}}})\big\}$ is  a prime generator group of $P_{\iota_{k}^{z}}$ and ${\rm int}(P_{\iota_{k}^{z}})$ is nonempty, this and  Lemma \ref{lemN2.1} imply that $x^*_{\iota_{k}^{z}j_1}$ and $x^*_{\iota_{k}^{z}j_2}$ are linearly independent. Hence ${\rm codim}(\mathcal{N}(x^*_{\iota_{k}^{z}j_1})\cap\mathcal{N}(x^*_{\iota_{k}^{z}j_2}))=2$. Noting that $X_1$ is a subspace of  $\mathcal{N}(x^*_{\iota_{k}^{z}j_1})\cap\mathcal{N}(x^*_{\iota_{k}^{z}j_2})$, it follows from (\ref{3.19})  that $X_2\cap\mathcal{N}(x^*_{\iota_{k}^{z}j_1})\cap\mathcal{N}(x^*_{\iota_{k}^{z}j_2})$, as a linear subspace of $X_2$, is of codimension 2. Hence,  by (\ref{3.36z}), $\hat P_{\iota_{k}^z}$ has   a face $\hat F$ such that
$z+\bar u+\lambda_{k}^z(\bar u'-\bar u)\in\hat  F$, ${\rm dim}(\hat F)\leq{\rm dim}(X_2)-2$
and so
$$z+[\bar u,\;\bar u']\subset\hat F-(\bar u+\lambda_{k}^z(\bar u'-\bar u))+[\bar u,\;\bar u']\subset\hat F+[\bar u-\bar u',\;\bar u'-\bar u].$$
Since each polyhedron has finitely many faces (cf. \cite{LYY, Ro}) , there exist finitely many linear subspaces $S_1,\cdots,S_l$ of $X_2$ and $\omega_1,\cdots,\omega_l\in X_2$ such that ${\rm dim}(S_j)\leq{\rm dom}(X_2)-2$ ($j=1,\cdots,l$) and
$z+[\bar u,\;\bar u']\subset \bigcup\limits_{j=1}^l(S_j+\omega_j+[\bar u-\bar u',\;\bar u'-\bar u])$ for all $z\in B_{X_3}(0,\delta)\setminus E_0$.
This means that $(B_{X_3}(0,\delta)\setminus E_0)+[\bar u,\;\bar u']\subset  \bigcup\limits_{j=1}^l(S_j+\omega_j+[\bar u-\bar u',\;\bar u'-\bar u])$ and so
$$\mu_{X_2}((B_{X_3}(0,\delta)\setminus E_0)+[\bar u,\;\bar u'])\leq\sum\limits_{j=1}^l\mu_{X_2}(S_j+\omega_j+[\bar u-\bar u',\;\bar u'-\bar u])=0.$$
Thus, by (\ref{3.23}) and the Fubini theorem, $\mu_{X_3}(B_{X_3}(0,\delta)\setminus E_0)=0$. Hence
$\mu_{X_3}(E_0)\geq\mu_{X_3}(B_{X_3}(0,\delta))>0$,
contradicting  $\mu_{X_3}(E_0)=0$.
This shows that (\ref{Jz3}) holds, that is, there exist $\bar z\in B_{X_3}(0,\delta)\setminus E_0$  and $(j_k^-,j_k)\in \overline{1\nu_{\iota_{k-1}^{\bar z}}}\times\overline{1\nu_{\iota_k^{\bar z}}}$ such that
$$
\bar x_k:=\bar z+\bar u+\lambda_k^{\bar z}(\bar u'-\bar u)\in F_{j_{k}^-}^\circ(\hat P_{\iota_{k-1}^{\bar z}})\cap F_{j_k}^\circ(\hat P_{\iota_{k}^{\bar z}})\quad\forall k\in \overline{1\gamma_{\bar z}}.
$$
Noting that $F^\circ_{j_{k}^-}(P_{\iota_{k-1}^{\bar z}})=X_1+F^\circ_{j_{k}^-}(\hat P_{\iota_{k-1}^{\bar z}})$ and
$F^\circ_{j_k}(P_{\iota_{k}^{\bar z}})=X_1+F^\circ_{j_k}(\hat P_{\iota_{k}^{\bar z}})$,
one has
$$\bar x_k\in F^\circ_{j_{k}^-}(P_{\iota_{k-1}^{\bar z}})\cap F^\circ_{j_k}(P_{\iota_{k}^{\bar z}})\quad\forall k\in\overline{1\gamma_{\bar z}}.$$
It follows from Lemma \ref{lem2.3} that for each $k\in\overline{1\gamma_{\bar z}}$,
$$\mathcal{N}_k:=\mathcal{N}(x^*_{\iota_{k-1}^{\bar z}j_{k}^-})=\mathcal{N}(x^*_{\iota_{k}^{\bar z}j_k})$$ and
$$F^\circ_{j_{k}^-}(P_{\iota_{k-1}^{\bar z}})\cap B_X(\bar x_k,r_k)=F^\circ_{j_k}(P_{\iota_{k}^{\bar z}})\cap B_X(\bar x_k,r_k)=(\bar x_k+\mathcal{N}_k)\cap B_X(\bar x_k,r_k)$$ for some $r_k>0$. Hence $(\bar x_k+\mathcal{N}_k)\cap B_X(\bar x_k,r_k)\subset P_{\iota_{k-1}^{\bar z}}\cap P_{\iota_{k}^{\bar z}}$ for all $k\in\overline{1\gamma_{\bar z}}$. This and  (\ref{L3.17}) imply that
$$T_{\iota_{k-1}^{\bar z}}|_{(\bar x_k+\mathcal{N}_k)\cap B_X(\bar x_k,r_k)}+b_{\iota_{k-1}^{\bar z}}=T_{\iota_{k}^{\bar z}}|_{(\bar x_k+\mathcal{N}_k)\cap B_X(\bar x_k,r_k)}+
b_{\iota_{k}^{\bar z}}\quad\forall k\in\overline{1\gamma_{\bar z}}.$$
Since $\mathcal{N}_k$ is a maximal subspace of $X$ and both $T_{\iota_{k-1}^{\bar z}}$ and $T_{\iota_{k}^{\bar z}}$ are linear,
$$T_{\iota_{k-1}^{\bar z}}|_{\mathcal{N}_k}=T_{\iota_{k}^{\bar z}}|_{\mathcal{N}_k}\quad\forall k\in\overline{1\gamma_{\bar z}}.$$
Noting that $X_1\subset \mathcal{N}_k$ (thanks to the definitions of $\mathcal{N}_k$ and $X_1$), it follows that $T_{\iota_{k-1}^{\bar z}}|_{X_1}=T_{\iota_{k}^{\bar z}}|_{X_1}$ for all $k\in\overline{1\gamma_{\bar z}}$, and so
$T_i|_{X_1}=T_{\iota_0^{\bar z}}|_{X_1}=T_{\iota_{\gamma_{\bar z}}^{\bar z}}|_{X_1}=T_{i'}|_{X_1}$
(thanks to (\ref{3.28})). The proof is complete.
\end{proof}

\begin{them}\label{thm3.1}
Let  $f\in\mathcal{PL}(X,Y)$, and let $X_1$ and $X_2$ be closed subspaces of $X$ such that
\begin{equation}\label{3.15}
\;\;X_1\subset X_f,\;\;X=X_1\oplus X_2,\;{\rm codim}(X_1)={\rm dim}(X_2)<\infty,
\end{equation}
where $X_f$ is as in Theorem \ref{thm3.0}. Then there exist $T\in\mathcal{L}(X_1,Y)$, a finite dimensional subspace $\hat Y$ of $Y$ and $g\in\mathcal{PL}(X_2,\hat Y)$ such that
\begin{equation}\label{Tg}
f(x_1+x_2)=T(x_1)+g(x_2)\quad\forall (x_1,x_2)\in X_1\times X_2.
\end{equation}
\end{them}

\begin{proof}
Take $(P_1,T_1,b_1),\cdots,(P_m,T_m,b_m)\in \mathcal{P}(X)\times\mathcal{L}(X,Y)\times Y$ such that (\ref{3.14}) and (\ref{L3.17}) hold. Then, by Theorem \ref{thm3.0},
\begin{equation}\label{T}
T_1(x)=\cdots=T_m(x)\quad\forall x\in X_f.
\end{equation}
Define $T:X_1\rightarrow Y$ as $T(x):=T_1(x)$ for all $x\in X_1$. Clearly, $T\in\mathcal{L}(X_1,Y)$.  For each $i\in\overline{1m}$, by (\ref{3.15}) and Proposition \ref{pro2.1},  take a polyhedron $\hat P_i$ in $X_2$ such that $P_i=X_1+\hat P_i$. It follows from (\ref{3.14}) that  $X_2=\bigcup\limits_{i=1}^m\hat P_i$. Moreover, since $X_1\subset X_f$,  (\ref{L3.17}) and (\ref{T}) imply  that
\begin{equation}\label{PL3.33}
\;\;\;\;f(x_1+x_2)=T_i(x_1+x_2)+b_i=T_i(x_1)+T_i(x_2)+b_i=T(x_1)+T_i(x_2)+b_i
\end{equation}
for all $(x_1,x_2)\in X_1\times \hat P_i$ and $i\in\overline{1m}$. Hence $f(x_2)=T_i(x_2)+b_i$ for all $i\in\overline{1m}$ and $x_2\in\hat P_i$. It follows  that
$$T_i(x_2)+b_i=T_{i'}(x_2)+b_{i'}\quad\forall i,i'\in\overline{1m}\;{\rm and}\;\forall x_2\in \hat P_{i}\cap\hat P_{i'}.$$
Let $\hat Y:={\rm span}\left(\bigcup\limits_{i=1}^m(T_i(X_2)+b_i)\right)$, and let $g:X_2\rightarrow\hat Y$ be such that
$$
g(x_2):=\left\{\begin{array}{ll}
T_1(x_2)+b_1, & {\rm if}\; x_2\in\hat P_1\\
\cdots\cdots\cdots\cdots & \cdots\cdots\cdots\cdot\\
T_m(x_2)+b_m, & {\rm if}\; x_2\in\hat P_m.
\end{array}
\right.
$$
Then $\hat Y$ is a subspace of $Y$ with ${\rm dim}(\hat Y)\leq m{\rm dim}(X_2)+1<+\infty$ and $g$ is well defined. It is easy  from (\ref{PL3.33}) to verify that (\ref{Tg}) holds.
\end{proof}

\section{Finite dimension reduction method to solve (PLP)}
In this section, with the help of Theorems \ref{thm3.0} and  \ref{thm3.1}, we reduce fully piecewise linear  problem (PLP) and linear subproblem (LP)$_i$ in the general normed space framework to the corresponding ones  in the finite-dimensional space framework.
Throughout this section, we assume that {\it  the objective function $f$ and all constraint functions $\varphi_j$ in (PLP) are completely known, that is, $P_i\in\mathcal{P}(X)$, $T_i\in\mathcal{L}(X,Y)$, $b_i\in Y$, $u_{ik}^*,x_{ij}^*\in X^*$, $b_i\in Y$ and $t_{ik},c_{ij}\in\mathbb{R}$ are known data such that (\ref{Poi}), (\ref{wp1'}) and  (\ref{wp2}) hold}. Throughout this section, {\it let $X_1$, $X_2$, $Z$, $\hat T$, $\Pi_Z$ and $C_Z$ be as in (\ref{X1}), (\ref{X2}), (\ref{Z}), (\ref{6.5}), (\ref{Pi}) and (\ref{CZ0}), respectively}.

To establish the main results in this section, we need the following lemma.

\begin{lem}\label{lem5.1}
Suppose that $(\hat T(X_1)\oplus Z)\cap {\rm int}(C)$ is nonempty. Then
\begin{equation}\label{NW4.26}
{\rm int}_Z(C_Z)=\Pi_Z((\hat T(X_1)\oplus Z)\cap {\rm int}(C)).
\end{equation}
\end{lem}

\begin{proof}
By the assumption, take $(\bar x_1, \bar z)\in X_1\times Z$ and $r>0$ such that
\begin{equation}\label{CZ}
\hat T(\bar x_1)+\bar z+rB_{\hat T(X_1)\oplus Z}\subset (\hat T(X_1)\oplus Z)\cap C.
\end{equation}
Noting that the projection  $\Pi_Z$ is an open mapping from $\hat T(X_1)\oplus Z$ to $Z$, (\ref{CZ0}) implies that ${\rm int}_Z(C_Z)\supset\Pi_Z((\hat T(X_1)\oplus Z)\cap {\rm int}(C))$. Hence it suffices to show the converse inclusion. To do this,
let $z\in{\rm int}_Z(C_Z)$. Then there exists  $\sigma>0$ sufficiently small such that $z+\sigma(z-\bar z)\in C_Z$, that is, $\hat T(x_1)+z+\sigma(z-\bar z)\in (\hat T(X_1)\oplus Z)\cap  C$ for some $x_1\in X_1$. It follows from (\ref{CZ}) and the convexity of $C$ that
\begin{eqnarray*}
\hat T\big(\frac{x_1+\sigma\bar x_1}{1+\sigma}\big)+z+\frac{\sigma rB_{\hat T(X_1)\oplus Z}}{1+\sigma}&=&\frac{\hat T(x_1)+z+\sigma(z-\bar z)}{1+\sigma}+\frac{\sigma(\hat T(\bar x_1)+\bar z+rB_{\hat T(X_1)\oplus Z})}{1+\sigma}\\
&\subset& (\hat T(X_1)\oplus Z)\cap C.
\end{eqnarray*}
Hence, by (\ref{CZ0}) and (\ref{Pi}),
$$C_Z\supset\Pi_Z\left(\hat T\big(\frac{x_1+\sigma\bar x_1}{1+\sigma}\big)+z+\frac{\sigma rB_{\hat T(X_1)\oplus Z}}{1+\sigma}\right)=z+\frac{\sigma r\Pi_Z(B_{\hat T(X_1)\oplus Z})}{1+\sigma}.$$
This shows that $z\in{\rm int}_Z(C_Z)$ (because $\Pi_Z$ is an open mapping from $\hat T(X_1)\oplus Z$ onto $Z$).
\end{proof}

Define
$\hat f:X_2\rightarrow Z$ as follows
$$\hat f(x_2):=(\Pi_Z\circ f)(x_2)=\Pi_Z(f(x_2))\quad\forall x_2\in X_2.$$
Then, $\hat f$ is a piecewise linear function between the finite dimensional spaces $X_2$ and $Z$. For each $i\in\overline{1m}$, recall that
$$\hat P_i:=\{x_2\in X_2:\,\langle u_{ik}^*,x_2\rangle\leq t_{ik}\;\;\forall k\in\overline{1q_i}\}.$$
By (\ref{Poi}), (\ref{wp1'}), (\ref{X1}) and  (\ref{X2}),  one has
$$X_2=\bigcup\limits_{i\in\overline{1m}}\hat P_i\;\;\;{\rm and}\;\;\;P_i=X_1+\hat P_i\;\;\forall i\in\overline{1m}.$$
From (\ref{wp2}),  (\ref{6.5}) and the definition of $Z$ (see (\ref{Z})), it is easy to verify that
$$f(X)=\bigcup\limits_{i\in\overline{1m}}(\hat T (X_1)+T_i(\hat P_i)+b_i)\subset \hat T(X_1)+Z=\hat T(X_1)\oplus Z$$
and $f(x_1+x_2)=\hat T(x_1)+f(x_2)\in \hat T(X_1)+f(x_2)$ for all $(x_1,x_2)\in  X_1\times X_2$.
Noting that  $f(x_2)-\hat f(x_2)=f(x_2)-\Pi_Z(f(x_2))\in \hat T(X_1)$, it follows that
\begin{equation}\label{NW4.25}
f(x_1+x_2)\in\hat T(X_1)+\hat f(x_2)\quad\forall (x_1,x_2)\in X_1\times X_2.
\end{equation}
To solve the original piecewise linear vector optimization problem (PLP), consider the following piecewise linear problem in the framework of finite dimensional spaces:
$$C_Z-\min\hat f(x_2)\;\;\;{\rm subject\;to}\;x_2\in X_2\;{\rm and}\;\varphi_1(x_2)\leq0,\cdots,\varphi_l(x_2)\leq0.\leqno{\rm (\widehat{PLP})}$$
Let $\hat A$ denote  the feasible set of $(\widehat{{\rm PLP}})$. Then  the feasible set $A$ of  (PLP) is equal to $X_1+\hat A$.

Next we establish the relationship between the weak Pareto optimal value set and weak Pareto solution set (resp. the Pareto solution set) of (PLP) and that  of ${\rm (\widehat{PLP})}$.

\begin{them}\label{thm5.1}
Let $S^w$ and $\hat S^w$ denote the  weak Pareto solution sets of piecewise linear problems ${\rm (PLP)}$ and ${\rm (\widehat{PLP})}$, respectively. The following statements hold:\\
(i) If
$(\hat T(X_1)\oplus Z)\cap{\rm int}(C)=\emptyset$ then ${\rm WE}(f(A),C)=f(A)$ and $S^w=A$.\\
(ii) If
$(\hat T(X_1)\oplus Z)\cap{\rm int}(C)\not=\emptyset$ then
\begin{equation}\label{NW4.27}
{\rm WE}(f(A),C)=\hat T(X_1)+{\rm WE}(\hat f(\hat A),C_Z)\;\;{\rm and}\;\; S^w=X_1+\hat S^w.
\end{equation}
\end{them}
\begin{proof}
First suppose that $(\hat T(X_1)\oplus Z)\cap{\rm int}(C)=\emptyset$. Then, since $\hat T(X_1)\oplus Z$ is a linear subspace of $Y$, $(\hat T(X_1)\oplus Z)\cap((\hat T(X_1)\oplus Z)-{\rm int}(C))=\emptyset$. Noting that
\begin{equation}\label{f(A)}
f(A)=f(X_1+\hat A)=\hat T(X_1)+\hat f(\hat A)\subset \hat T(X_1)\oplus Z
\end{equation}
(thanks to (\ref{NW4.25})), it follows that  $f(A)\cap(f(A)-{\rm int}(C))=\emptyset$. This shows that  ${\rm WE}(f(A),C)=f(A)$ and $S^w=A$.
Next suppose that $(\hat T(X_1)\oplus Z)\cap{\rm int}(C)\not=\emptyset$. Then, by Lemma \ref{lem5.1}, ${\rm int}_Z(C_Z)=\Pi_Z(\hat T(X_1)\oplus Z)\cap{\rm int}(C))$.
Since $\Pi_Z$ is the projection from $\hat T(X_1)\oplus Z$ to $Z$, $\hat T(X_1)+E=\hat T(X_1)+\Pi_Z(E)$ for any set $E$ in $\hat T(X_1)\oplus Z$. Hence, by Lemma \ref{lem5.1},
\begin{eqnarray*}
\hat T(X_1)+(\hat T(X_1)\oplus Z)\cap{\rm int}(C)&=&\hat T(X_1)+\Pi_Z((\hat T(X_1)\oplus Z)\cap{\rm int}(C))\\
&=&\hat T(X_1)+{\rm int}_Z(C_Z).
\end{eqnarray*}
Noting that ${\rm WE}(\Omega,C)=\Omega\setminus (\Omega+{\rm int}(C))$ for any set in $Y$, it follows from (\ref{f(A)}) that
\begin{eqnarray*}
{\rm WE}(f(A),C)
&=&(\hat T(X_1)+\hat f(\hat A))\setminus(\hat T(X_1)+\hat f(\hat A)+{\rm int}(C))\\
&=&(\hat T(X_1)+\hat f(\hat A))\setminus(\hat T(X_1)+\hat f(\hat A)+(\hat T(X_1)\oplus Z)\cap{\rm int}(C))\\
&=&(\hat T(X_1)+\hat f(\hat A))\setminus(\hat T(X_1)+\hat f(\hat A)+{\rm int}_Z(C_Z)).
\end{eqnarray*}
Since $\hat f(\hat A)\subset Z$ and $\hat T(X_1)\cap Z=\{0\}$,
$$(\hat T(X_1)+\hat f(\hat A))\setminus(\hat T(X_1)+\hat f(\hat A)+{\rm int}_Z(C_Z))=\hat T(X_1)+\hat f(\hat A)\setminus(\hat f(\hat A)+{\rm int}_Z(C_Z)).$$
Hence
$$
{\rm WE}(f(A),C)=\hat T(X_1)+\hat f(\hat A)\setminus(\hat f(\hat A)+{\rm int}_Z(C_Z))=\hat T(X_1)+{\rm WE}(\hat f(\hat A),C_Z).
$$
This shows  the first equality of (\ref{NW4.27}). To prove the second equality of (\ref{NW4.27}), let $x_2\in \hat S^w$. Then $x_2\in\hat A$ and $\hat f(x_2)\in {\rm WE}(\hat f(\hat A),C_Z)$. Hence,
$$X_1+x_2\subset X_1+\hat A=A\;{\rm and}\; f(X_1+x_2)=\hat T(X_1)+\hat f(x_2)\subset {\rm WE}(f(A),C)$$
(thanks to (\ref{NW4.25}) and the first equality of (\ref{NW4.27})). It follows that $X_1+x_2\subset S^w$ and so $X_1+\hat S^w\subset S^w$. Conversely, let $x\in S^w$. Then there exists $(x_1,x_2)\in X_1\times \hat A$ such that $x=x_1+x_2$ and $f(x_1+x_2)\in {\rm WE}(f(A),C)=\hat T(X_1)+{\rm WE}(\hat f(\hat A),C_Z)$.
Noting that $f(x_1+x_2)\in f(X_1+x_2)=\hat T(X_1)+\hat f(x_2)$ and $\hat T(X_1)\cap Z=\{0\}$, one has  $\hat f(x_2)\in {\rm WE}(\hat f(\hat A),C_Z)$. Hence $x_2\in \hat S^w$ and $x=x_1+x_2\in X_1+\hat S^w$.  Hence the second equality of (\ref{NW4.27}) holds. The proof is complete.
\end{proof}

\begin{them}\label{thm5.2}
Let $(x_1,x_2)\in X_1\times \hat A$. Then $f(x_1+x_2)\in {\rm E}(f(A),C)$ if and only if
$\hat f(x_2)\in {\rm E}(\hat f(\hat A),C_Z)$ and $C_Z=C\cap(\hat T(X_1)\oplus Z)$.
\end{them}

\begin{proof}
By (\ref{NW4.25}),  $f(A)=f(X_1+\hat A)=\hat T(X_1)+\hat f(\hat A)$.
Hence
$$f(A)-f(x_1+x_2)=\hat T(X_1)+\hat f(\hat A)-\hat f(x_2).$$
Noting that $\hat f(\hat A)-\hat f(x_2)\subset Z$, it follows that
$$(f(A)-f(x_1+x_2))\cap -C=(\hat T(X_1)+\hat f(\hat A)-\hat f(x_2))\cap- (C\cap(\hat T(X_1)\oplus Z)).$$
Thus, from the definitions of the projection $\Pi_Z: \hat T(X_1)\oplus Z\rightarrow Z$ (see (\ref{Pi})), it is easy to verify that
$$(f(A)-f(x_1+x_2))\cap- C=\Pi_1(C\cap(\hat T(X_1)\oplus Z))+(\hat f(\hat A)-\hat f(x_2))\cap -C_Z,$$
where $\Pi_1(y+z)=y$ for all $(y,z)\in \hat T(X_1)\oplus Z$.
Therefore, $f(x_1+x_2)\in {\rm E}(f(A),C)$ is equivalent to
$$\Pi_1(C\cap(\hat T(X_1)\oplus Z))+(\hat f(\hat A)-\hat f(x_2))\cap -C_Z=\{0\}.$$
Since $\hat T(X_1)\cap Z=\{0\}$, it follows that $f(x_1+x_2)\in {\rm E}(f(A),C)$ if and only if
$$\Pi_1(C\cap(\hat T(X_1)\oplus Z))=(\hat f(\hat A)-\hat f(x_2))\cap-C_Z=\{0\},$$
namely $C_Z=C\cap(\hat T(X_1)\oplus Z)$ and $\hat f(x_2)\in E(\hat f(\hat A),C_Z)$. The proof is complete.
\end{proof}

The following corollary is a consequence of Theorem \ref{thm5.2}.

\begin{coro}\label{coro5.1}
Let $\hat S$ denote the Pareto solution set of piecewise linear problem ${\rm (\widehat{PLP})}$. The following statements hold:\\
(i) If $C_Z\not=C\cap(\hat T(X_1)\oplus Z)$ then $S=\emptyset$.\\
(ii) If $C_Z=C\cap(\hat T(X_1)\oplus Z)$ then
$$S=X_1+\hat S\;\;{\rm and}\;\;{\rm E}(f(A),C)=\hat T(X_1)+{\rm E}(\hat f(\hat A),C_Z).$$

{\rm {\bf Remark.}} By Corollary \ref{coro5.1}(i) and Theorem \ref{thm5.1}(i), piecewise linear problem \end{coro}
{\it (PLP) has no Pareto solution when $C_Z\not=C\cap(\hat T(X_1)\oplus Z)$, and the weak Pareto solution set of (PLP) is just the entire feasible set $A$ of (PLP) when $(\hat T(X_1)\oplus Z)\cap{\rm int}(C)=\emptyset$. Therefore, we only need to consider
the Pareto solution set  and the weak Pareto solution of (PLP) when  $C_Z=C\cap(\hat T(X_1)\oplus Z)$ and $(\hat T(X_1)\oplus Z)\cap{\rm int}(C)\not=\emptyset$, respectively.}

In the framework of finite dimensional spaces, for $i\in\overline{1m}$, we consider the following  linear subproblem
$$C_Z-\min \Pi_Z(T_ix+b_i)\;\;{\rm subject\;to}\;x\in \hat A_i.\leqno{{\rm (\widehat{LP})}_{i}},$$
where  $\hat A_i$ is as (\ref{hatA}).

By Theorem \ref{thm5.1} and Corollary \ref{coro5.1} (with linear problems (LP)$_i$ and ${\rm (\widehat{LP})}_i$ replacing respectively piecewise linear problems (PLP) and ${\rm (\widehat{PLP})}$), we have the following result (thanks to $A_i=X_1+\hat A_i$).

\begin{pro}\label{pro5.1}
For each $i\in\overline{1m}$, let $S_i$ (resp. $S_i^w$) and $\hat S_i$ (resp. $\hat S^w_i$) denote the Pareto solution sets (resp. weak Pareto solution sets) of linear problem (LP)$_i$ and ${\rm (\widehat{LP})}_i$, respectively. The following statements hold:\\
(i) $S_i=\emptyset$ if $C_Z\not=C\cap(\hat T(X_1)\oplus Z)$.\\
(ii) $S_i=X_1+\hat S_i$ if $C_Z=C\cap(\hat T(X_1)\oplus Z)$.\\
(iii) $S_i^w=A_i$ if $(\hat T(X_1)\oplus Z)\cap{\rm int}(C)=\emptyset$.\\
(iv) $S_i^w=X_1+\hat S_i^w$ if $(\hat T(X_1)\oplus Z)\cap{\rm int}(C)\not=\emptyset$.
\end{pro}

The following theorem provides exact formulas for the weak Pareto solution set and weak Pareto optimal value set for  piecewise linear problem ${\rm (\widehat{PLP})}$.

\begin{them}\label{pro5.2}
For each $i\in \overline{1m}$, let
\begin{equation}\label{4.17}
\hat V^w_{i}:=\Pi_Z\left(T_i(\hat A_{i})+b_i\right)\setminus\left(\hat f(\hat A)+{\rm int}_Z(C_Z)\right)
\end{equation}
and
\begin{equation}\label{n4.18}
\breve{S}_i:=\hat A_i\cap (\Pi_Z\circ T_i)^{-1}\left(\hat V^w_i-\Pi_Z(b_i)\right).
\end{equation}
Suppose that $(\hat T(X_1)\oplus Z)\cap{\rm int}(C)\not=\emptyset$. Then the following statements hold:\\
(i) $\hat S^w=\bigcup\limits_{i\in\bar I}\breve{S}_i$ and ${\rm WE}\left(\hat f(\hat A),C_Z\right)=\bigcup\limits_{i\in\bar I}\hat V^w_i$, where $I:=\{i\in\overline{1m}:\;\hat A_i\not=\emptyset\}$.\\
(ii) If, in addition, the ordering cone $C$ in $Y$ is assumed to be polyhedral, then for each $i\in\bar I$ there exist finitely many polyhedra $\hat P_{i1},\cdots,\hat P_{iq_i}$ in $X_2$ and faces $\hat F_{i1},\cdots,\hat F_{iq_i}$ of $\hat A_i$ such that
$\breve{S}_i=\bigcup\limits_{j=1}^{q_i}\hat P_{ij}$
and $\hat P_{ij}\subset \hat F_{ij}\subset\hat S^w_i$ for all $j\in\overline{1q_i}$.
Consequently, $\hat S^w$ is the union of finitely many polyhedra in $X_2$, each one of which is contained in a weak Pareto face of some linear subproblem ${\rm (\widehat{LP})}_i$.
\end{them}

\begin{proof}
Let $i$ be an arbitrary element in $\bar I$. Since $\hat f(\hat x)=\Pi_Z(T_i(\hat x))+\Pi_Z(b_i)$ for all $\hat x\in\hat A_i$,
$$\hat A_i\cap(\Pi_Z\circ T_i)^{-1}(\hat V^w_i-\Pi_Z(b_i))=\hat A_i\cap\hat f^{-1}(\hat V^w_i).$$
Hence, by (\ref{4.17}) and (\ref{n4.18}),
\begin{equation}\label{C1}
\breve S_i=\hat A_i\cap \hat f^{-1}(\hat V^w_i)\;\;\;{\rm and}\;\;\;\hat V^w_i=\hat f(\breve S_i).
\end{equation}
Thus, to prove (i),
it suffices to show that $\breve S_i=\hat S^w\cap \hat A_i$
(because $\hat A=\bigcup\limits_{i\in\bar I}\hat A_i$ and $\hat f(\hat S^w)={\rm WE}(\hat f(\hat A),C_Z)$).
To do this, let $\hat a_i\in\hat S^w\cap\hat A_i$. Then $\hat f(\hat a_i)\in {\rm WE}(\hat f(\hat A),C_Z)$, that is, $\hat f(\hat a_i)\not\in \hat f(\hat A)+{\rm int}_Z(C_Z)$. Since
$$\hat f(\hat a_i)=\Pi_Z(T_i(\hat a_i)+b_i)\in \Pi_Z(T_i(\hat A_{i})+b_i),$$
this and (\ref{4.17}) imply that $\hat f(\hat a_i)\in\hat V^w_i$. Hence $\hat a_i\in\breve S_i$ (thanks to  (\ref{C1})).
This shows that $\hat  S^w\cap\hat A_i \subset \breve S_i$.
Conversely, let $\hat a_i\in \breve S_i$. Then, by (\ref{n4.18}), $\hat a_i\in\hat A_i$, $\Pi_Z(T_i(\hat a_i))\in \hat V^w_i-\Pi_Z(b_i)$ and so  $\hat f(\hat a_i)\in\hat  V^w_i$. Since $\hat f(\hat a_i)\in\hat f(\hat A_i)=\Pi_Z\big(T_i(\hat A_{i})+b_i\big)$, $\hat f(\hat a_i)\not\in \hat f(\hat A)+{\rm int}_Z(C_Z)$ (thanks to (\ref{4.17})). Noting that $\hat A_i\subset\hat A$, it follows that
$$\hat f(\hat a_i)\in \hat f(\hat A)\setminus\big(\hat f(\hat A)+{\rm int}_Z(C_Z)\big)={\rm WE}(\hat f(\hat A),C_Z),$$
and so
$\hat a_i\in \hat A_i\cap \hat f^{-1}({\rm WE}(\hat f(\hat A),C_Z))=\hat A_i\cap\hat  S^w$. This shows that $\breve S_i\subset \hat A_i\cap\hat  S^w$. Therefore, $\breve S_i=\hat A_i\cap\hat  S^w$. The proof of (i) is complete.

To prove (ii), suppose that the ordering cone $C$ is polyhedral. Then, since the projection mapping $\Pi_Z:\hat T(X_1)\oplus Z\rightarrow Z$ is a linear operator and since $Z$ is finite dimensional, $C_Z=\Pi_Z((\hat T(X_1)\oplus Z)\cap C)$ is a polyhedral cone in $Z$ (thanks to \cite[Theorem 19.3]{Ro} and Proposition 2.1). On the other hand, by the assumption that $(\hat T(X_1)\oplus Z)\cap {\rm int}(C)\not=\emptyset$, Lemma \ref{lem5.1} implies that
$${\rm int}_Z(C_Z)=\Pi_Z((\hat T(X_1)\oplus Z)\cap {\rm int}(C))\not=\emptyset.$$
Since $\Pi_Z(T_j(\hat A_j)+b_j)$ and $C_Z$ are polyhedra in the finite dimensional space $Z$, their sum $\Pi_Z(T_j(\hat A_j)+b_j)+C_Z$ is a polyhedron in $Z$ and so is closed. Hence
$$\Pi_Z(T_j(\hat A_j)+b_j)+C_Z={\rm cl}\big(\Pi_Z(T_j(\hat A_j)+b_j)+{\rm int}_Z(C_Z)\big).$$
Noting that $\Pi_Z(T_j(\hat A_j)+b_j)+{\rm int}_Z(C_Z)$ is open in $Z$,  it follows  that
$$
{\rm int}_Z\big(\Pi_Z(T_j(\hat A_j)+b_j)+C_Z\big)=\Pi_Z(T_j(\hat A_j)+b_j)+{\rm int}_Z(C_Z).
$$
Thus, by Proposition \ref{pro2.1}, there exist $(z_{j1}^*,r_{j1}),\cdots,(z_{jq_j}^*,r_{jq_j})$ in $Z^*\times\mathbb{R}$ such that
\begin{equation}\label{5.12}
\Pi_Z(T_j(\hat A_j)+b_j)+{\rm int}_Z(C_Z)=\{z\in Z:\;\langle z_{jk}^*,z\rangle<r_{jk},\;k=1\cdots,q_j\}.
\end{equation}
Since $\hat A=\bigcup\limits_{j\in\bar I}\hat A_j$, it follows from (\ref{4.17}) that
\begin{eqnarray*}
\hat V^w_i&=&\Pi_Z(T_i(\hat A_i)+b_i)\setminus\left(\bigcup\limits_{j\in\bar I}(\hat f(\hat A_j)+{\rm int}_Z(C_Z)\right)\\
&=&\Pi_Z(T_i(\hat A_i)+b_i)\setminus\left(\bigcup\limits_{j\in\bar I}(\Pi_Z(T_j(\hat A_j)+b_j)+{\rm int}_Z(C_Z)\right)\\
&=&\Pi_Z(T_i(\hat A_i)+b_i)\setminus\left(\bigcup\limits_{j\in\bar I}\bigcap\limits_{k=1}^{q_j}\{z\in Z:\;\langle z_{jk}^*,z\rangle<r_{jk}\}\right)\\
&=&\bigcap\limits_{j\in\bar I}\bigcup\limits_{k=1}^{q_j}\left(\Pi_Z(T_i(\hat A_i)+b_i)\setminus\{z\in Z:\;\langle z_{jk}^*,z\rangle<r_{jk}\}\right)\\
&=&\bigcap\limits_{j\in\bar I}\bigcup\limits_{k=1}^{q_j}\left(\Pi_Z(T_i(\hat A_i)+b_i)\cap\{z\in Z:\;\langle z_{jk}^*,z\rangle\geq r_{jk}\}\right).
\end{eqnarray*}
Since $\bar I$ is a subset of $\overline{1m}$, we assume without loss of generality that there exists $n\in\overline{1m}$ such that $\bar I=\overline{1n}$. For any $(k_1,\cdots,k_n)\in \overline{1q_1}\times\cdots\times \overline{1q_n}$, let
$$Q^i_{(k_1,\cdots,k_n)}:=\bigcap\limits_{j=1}^n\left(\Pi_Z(T_i(\hat A_i)+b_i)\cap\{z\in Z:\;\langle z_{jk_j}^*,z\rangle\geq r_{jk_j}\}\right).$$
Then, each $Q^i_{(k_1,\cdots,k_n)}$ is a polyhedron in $Z$ and
\begin{equation}\label{5.13}
\hat V^w_i=\bigcup\limits_{(k_1,\cdots,k_n)\in\Pi_i}Q^i_{(k_1,\cdots,k_n)},
\end{equation}
where $\Pi_i:=\big\{(k_1,\cdots,k_n)\in \overline{1q_1}\times\cdots\times \overline{1q_n}:\; Q^i_{(k_1,\cdots,k_n)}\not=\emptyset\big\}$.
Let
$$\hat P^i_{(k_1,\cdots,k_n)}:=\hat A_i\cap (\Pi_Z\circ T_i)^{-1}(Q^i_{(k_1,\cdots,k_n)}-\Pi_Z(b_i))\quad\forall (k_1,\cdots,k_n)\in\Pi_i.$$
Then each $\hat P^i_{(k_1,\cdots,k_n)}$ is a polyhedron in the finite dimensional space $X_2$ and
\begin{equation}\label{ii5.17}
\breve S_i=\hat A_i\cap(\Pi_Z\circ T_i)^{-1}(\hat V^w_i-\Pi_Z(b_i))=\bigcup\limits_{(k_1,\cdots,k_n)\in\Pi_i}\hat P^i_{(k_1,\cdots,k_n)}.
\end{equation}
Thus, to prove (ii), it suffices to show  that for each $(k_1,\cdots,k_n)\in\Pi_i$ there exists a face $\hat F$ of $\hat A_i$ such that
$\hat P^i_{(k_1,\cdots,k_n)}\subset\hat F\subset \hat S^w_i$. By
Theorem ABB (applied to linear problem ${\rm (\widehat{LP})}_i$), there exist finitely many faces $\hat F_{i1}\cdots, \hat F_{i\nu_i}$ of $\hat A_i$ such that $\hat S^w_i=\bigcup\limits_{j=1}^{\nu_i}\hat F_{ij}$. Noting that each $\hat P^i_{(k_1,\cdots,k_n)}$ is contained in $\hat S^w_i$ (thanks to (i) and (\ref{ii5.17})), it follows from Proposition \ref{pro2.5} that $\hat P^i_{(k_1,\cdots,k_n)}\subset\hat F_{ij'}$ for some $j'\in\overline{1\nu_i}$. The proof is complete.
\end{proof}

Formulas (i) and (ii) in Step 3 of the procedure provided in Section 1 are immediate from Theorems \ref{thm5.1} and \ref{pro5.2}. \\

To establish formulas for the Pareto solution set and Pareto optimal value set,  we need the following lemma, which is a  variant of a formula appearing in the proof of \cite[Theorem 3.4]{Zh}.
\begin{lem}\label{lem5.2}
Let $B_1,\cdots,B_m$ be subsets of $Y$. Then
$${\rm E}\left(\bigcup\limits_{i\in\overline{1m}}B_i,C\right)=\bigcup\limits_{i\in\overline{1m}}\bigcap\limits_{j\in\overline{1m}}\big({\rm E}(B_i,C)\setminus((B_j+C)\setminus  {\rm E}(B_j,C))\big).$$
\end{lem}

\begin{proof}
Let $B:=\bigcup\limits_{i\in\overline{1m}}B_i$ and
$E_i:=\bigcap\limits_{j\in\overline{1m}}\big({\rm E}(B_i,C)\setminus((B_j+C)\setminus  {\rm E}(B_j,C))\big)$ for all $i\in \overline{1m}$.
We need to show ${\rm E}(B,C)=\bigcup\limits_{i=1}^mE_i$.
For each $y'\in {\rm E}(B,C)$, there exists $i'\in\overline{1m}$ such that $y'\in B_{i'}$ and so  $y'\in E(B_{i'},C)$. Since $(B_j+C)\cap {\rm E}(B,C)\subset{\rm E}(B_j,C)$ for all $j\in\overline{1m}$,
$y'\in {\rm E}(B_j,C)\;{\rm for\;all}\; j\in\overline{1m}\;{\rm with}\;y'\in B_j+C$.
It follows that $y'\not\in (B_j+C)\setminus {\rm E}(B_j,C)$ for all $j\in \overline{1m}.$
Hence $y'\in E(B_{i'},C)\setminus((B_j+C)\setminus {\rm E}(B_j,C)))$ for all $j\in \overline{1m}$, that is, $y'\in E_{i'}$. This shows that ${\rm E}(B,C)\subset\bigcup\limits_{i\in\overline{1m}}E_i$.
Conversely,  let $y\in
\bigcup\limits_{i=1}^mE_i$. Then there exists $i_0\in\overline{1m}$ such that $y\in E_{i_0}$. Let $z\in B\cap(y-C)$. We only need to show $z=y$. Take  $j\in \overline{1m}$ such that $z\in B_j$. It follows that
$z\in B_j\cap (y-C)$.
Noting that $E_{i_0}\subset{\rm E}(B_{i_0},C)$, it
is clear that $z=y$ if $j=i_0$. Now suppose that $j\not=i_0$. By the
definition of $E_{i_0}$, one has $y\in E(B_{i_0},C)\setminus((B_j+C)\setminus{\rm E}(B_j,C))$, and so $y\not \in (B_j+C)\setminus{\rm E}(B_j,C)$.
Since $y\in z+C\subset B_j+C$, $y\in {\rm E}(B_j,C)$. Hence $\{y\}=B_j\cap (y-C)\ni z$. This shows that $y=z$. The proof is complete.
\end{proof}

\begin{pro}\label{pro5.3}
Let $\hat S$ and $\hat S_i$ ($i\in\bar I:=\{i\in\overline{1m}:\,\hat A_i\not=\emptyset\}$) denote the Pareto solution set of piecewise linear problem ${\rm (\widehat{PLP})}$ and linear subproblem ${\rm (\widehat{LP})}_i$, respectively. Suppose that the ordering cone $C$ is polyhedral. Then there exist finitely many generalized polyhedra $\hat F_1,\cdots,\hat F_p$ in $X_2$ such that the following statements hold:\\
(i) $\hat S=\bigcup\limits_{k=1}^p\hat F_k$.\\
(ii) For each $k\in \overline{1p}$ there exist $i\in\bar I$ and a face $\hat F$ of $\hat A_i$ such that $\hat F_k\subset \hat F\subset \hat S_i$.
\end{pro}

\begin{proof}
For each $i\in\bar I$, let $\tilde S_i:=\hat A_i\cap \hat S$. Then $\hat S=\bigcup\limits_{i\in\bar I}\tilde S_i$, and  $\tilde S_i$ is clearly contained in the Pareto solution set $\hat S_i$ of linear subproblem ${\rm (\widehat{LP})}_i$. Thus, by Theorem ABB and Proposition 2.4, it suffices to show that there exist finitely many generalized polyhedra  $\hat G_{i1},\cdots,\hat G_{i\nu_i}$ in $X_2$ such that $\tilde S_i=\bigcup\limits_{k=1}^{\nu_i}\hat G_{ik}$.  Noting that $\hat f|_{\hat A_i}=\Pi_Z\circ f|_{\hat A_i}=\Pi_Z\circ T_i|_{\hat A_i}+\Pi_Z(b_i)$,
one has
\begin{equation}\label{5.18}
\;\;\tilde S_i=\hat A_i\cap \hat f^{-1}({\rm E}(\hat f(\hat A),C_Z))= \hat A_i\cap (\Pi_Z\circ T_i)^{-1}({\rm E}(\hat f(\hat A),C_Z)-\Pi_Z(b_i)).
\end{equation}
Since $C$ is a polyhedral cone in $Y$, $C\cap (\hat T(X_1)\oplus Z)$ is a polyhedral cone in $\hat T(X_1)\oplus Z$. Hence $C_Z=\Pi_Z(C\cap (\hat T(X_1)\oplus Z))$ is a polyhedral cone in the finite dimensional space $Z$. It follows  that $B_j+C_Z$ is a polyhedron in $Z$ and ${\rm E}(B_j,C_Z)={\rm E}(B_j+C_Z,C_Z)$ is the union of finitely many polyhedra in $Z$ for  each $j\in\bar I$ (thanks to Theorem ABB), where $B_j:=\Pi_Z(T_j(\hat A_j)+b_j)$. Hence $E_i:=\bigcap\limits_{j\in \bar I}{\rm E}(B_i,C_Z)\setminus\big(B_j+C_Z)\setminus{\rm E}(B_j,C_Z)\big)$
is the union of finitely many generalized polehedra in $Z$ for all $i\in \bar I$. Since
$$\hat f(\hat A)=\bigcup\limits_{i\in\bar I}\hat f(\hat A_i)=\bigcup\limits_{i\in\bar I}B_i,$$
This and Lemma \ref{lem5.2} imply that ${\rm E}(\hat f(\hat A),C_Z)=\bigcup\limits_{i\in\bar I}E_i$ and so  ${\rm E}(\hat f(\hat A),C_Z)$ is the union of finitely many generalized polyhedra in $Z$. Thus, by (\ref{5.18}), for each $i\in\bar I$ there exist finitely many generalized polyhedra $\hat G_{i1},\cdots,\hat G_{i\nu_i}$ in $X_2$ such that $\tilde S_i=\bigcup\limits_{k=1}^{\nu_i}\hat G_{ik}$.  The proof is complete.
\end{proof}

Based on Corollary \ref{coro5.1} and Proposition \ref{pro5.3} (and its proof), we can establish exact formulas for the Pareto solution set and Pareto optimal value set of  fully piecewise linear vector optimization problem (PLP).

The following corollary establishes the structure of the weak Pareto solution set and Pareto solution set for (PLP).

\begin{coro}\label{cor4.2}
Let $S^w$ and $S$ be the weak Pareto solution set and Pareto solution set of fully  piecewise linear vector optimization problem (PLP), respectively.
Suppose that the ordering cone $C$ is polyhedral. Then the following statements hold:\\
(i) There exist finitely many polyhedra $F_1,\cdots,F_p$ in $X$ such that  $S^w=\bigcup\limits_{k=1}^pF_k$ and  each $F_k$  is contained in a weak Pareto face of some linear subproblem (LP)$_{i}$.\\
(ii) There exist finitely many generalized polyhedra $F_1,\cdots,F_p$ in $X$ such that  $S=\bigcup\limits_{k=1}^pF_k$ and  $F_k$  is contained in a Pareto face of some linear subproblem (LP)$_{i}$.
\end{coro}

Corollary \ref{cor4.2} is immediate from Theorem \ref{thm5.1}, Corollary \ref{coro5.1}, Propositions \ref{pro5.1}, \ref{pro5.3} and \ref{pro5.2} and Corollary \ref{pro2.4}.

Dropping the polyhedral assumption on the ordering cone $C$ but imposing the $C$-convexity assumption on $f(A)$, we have the following structure theorem on  the weak Pareto solution set  of (PLP), which generalizes and improves the corresponding result established by Arrow et al. \cite{ABB} in the finite-dimension and linear case.

\begin{them}\label{thm4.2}
Let $C$ be a convex cone in $Y$ such that $f(A)$ is $C$-convex, that is, $f(A)+C$ is a convex subset of $Y$. Then there exist finitely many polyhedra $F_1,\cdots,F_p$ in $X$ satisfying the following properties:\\
(i) $S^w=\bigcup\limits_{k=1}^pF_k$.\\
(ii) For each $k$ there exists $i\in\bar I$ such that $F_k$ is a face of $A_i$ and $F_k\subset S^w_{i}$, where $\bar I:=\{i\in \overline{1m}:\,A_i\not=\emptyset\}$ and $S^w_{i}$ is the weak Pareto solution set of linear subproblem (LP)$_{i}$.\\
Consequently each $F_k$ is just a weak Pareto face of linear subproblem (LP)$_{i}$ for some $i\in\bar I$.
\end{them}

\begin{proof}
Let $x\in A$. Then  $x\in S^w$ if and only if $f(A)\cap (f(x)-{\rm int}(C))=\emptyset$, which is equivalent to $(f(A)+C)\cap(f(x)-{\rm int}(C))=\emptyset$. Thus, by the separation theorem and the convexity of
$f(A)+C$, $x\in S^w$ if and only if
there exists $c^*\in C^+\setminus\{0\}$ such that $\langle c^*,f(x)\rangle=\inf\limits_{u\in A}\langle c^*,f(u)\rangle$.
Let
$$S^w(c^*):=\left\{x\in A:\;\langle c^*,f(x)\rangle=\inf\limits_{u\in A}\langle c^*,f(u)\rangle\right\}\quad\forall c^*\in C^+\setminus\{0\}$$
and $C^+(f,A):=\big\{c^*\in C^*\setminus\{0\}:\;S^w(c^*)\not=\emptyset\big\}$.
Then, since the feasible set $A$ of (PLP) is equal to $\bigcup\limits_{i\in\overline{1m}}A_i$,  one has $S^w=\bigcup\limits_{c^*\in C^+(f,A)}S^w(c^*)=\bigcup\limits_{c^*\in C^+(f,A)}\bigcup\limits_{i\in \Lambda(c^*)}S^w(c^*)\cap A_{i}$,
where $\Lambda(c^*):=\{i\in \bar I:\;S^w(c^*)\cap A_{i}\not=\emptyset\}$.
On the other hand, for $c^*\in C^+(f,A)$ and $i\in \Lambda(c^*)$,
\begin{eqnarray*}
S^w(c^*)\cap A_{i}&=&\{x\in A_{i}:\;\langle c^*,f(x)=\min\limits_{u\in A_{i}}\langle c^*,f(u)\rangle\}\\
&=&\{x\in A_{i}:\;\langle c^*,T_ix+b_i\rangle=\min\limits_{u\in A_{i}}\langle c^*,T_iu+b_i\rangle\}\\
&=&\{x\in A_{i}:\;\langle c^*,T_ix\rangle=\min\limits_{u\in A_{i}}\langle c^*,T_iu\rangle\}\\
&=&\{x\in A_{i}:\;\langle T^*_i(c^*),x\rangle=\min\limits_{u\in A_{i}}\langle T^*_i(c^*),u\rangle\}
\end{eqnarray*}
(thanks to (\ref{wp2}) and (\ref{wp2'})) is a face of $A_{i}$ and a subset of the weak Pareto solution set of linear subproblem (LP)$_{i}$. Since  every polyhedron only has finitely many faces, there exist $c_1^*,\cdots,c_p^*\in C^+(f,A)$ such that
$$S^w=\bigcup\limits_{c^*\in C^+(f,A)}\bigcup\limits_{i\in \Lambda(c^*)}S^w(c^*)\cap A_{i}=\bigcup\limits_{k=1}^p\bigcup\limits_{i\in \Lambda(c_k^*)}S^w(c_k^*)\cap A_{i}.$$
The proof is complete.
\end{proof}

{\bf Remark.}  If $Y=\mathbb{R}$ and $C=\mathbb{R}_+$, then each set in $Y$ is trivially $C$-convex. Moreover, if $f$ is $C$-convex (i.e. ${\rm epi}_C(f)=\{(x,y):\;y\in f(x)+C\}$ is convex) then $f(A)$ is $C$-convex.

\end{document}